\documentclass{article}

\usepackage{makeidx}
\usepackage{amssymb}
\usepackage{amsfonts}
\usepackage{amsmath}
\usepackage{array}
\usepackage{amsthm}
\usepackage[utf8]{inputenc}
\usepackage{geometry}
\usepackage{lipsum}
\usepackage{nicefrac}
\usepackage{graphicx}
\usepackage{url}
\usepackage{algpseudocode}
\usepackage{algorithm}
\usepackage{subcaption}

\newtheorem{theorem}{Theorem}
\newtheorem{corollary}{Corollary}

\theoremstyle{remark}
\newtheorem{remark}{Remark}

\newcommand{\footremember}[2]{%
    \footnote{#2}
    \newcounter{#1}
    \setcounter{#1}{\value{footnote}}%
}
\newcommand{\footrecall}[1]{%
    \footnotemark[\value{#1}]%
}

\newcommand{\bB}{\bar{B}}

\newcommand{\da}{\partial_{\alpha}}
\newcommand{\ddaa}{\partial^2_{\alpha}}
\newcommand{\db}{\partial_{\beta}}

\newcommand{\eps}{\varepsilon}
\newcommand{\heps}{\bar{\varepsilon}}

\newcommand{\hphidp}{\widehat{\phi}_{\text{\tiny DP}}}
\newcommand{\hpsidp}{\widehat{\psi}_{\text{\tiny DP}}}
\newcommand{\hfdp}{\widehat{f}_{\text{\tiny DP}}}

\newcommand{\Kry}{\mathcal{K}}

\newcommand{\mcE}{\mathcal{E}}
\newcommand{\mcG}{\mathcal{G}}
\newcommand{\mcQ}{\mathcal{Q}}
\newcommand{\mcR}{\mathcal{R}}
\newcommand{\mcL}{L}

\newcommand{\phidp}{\phi_{\text{\tiny DP}}}
\newcommand{\phiR}{\phi_{\text{\tiny R}}}
\newcommand{\phiqo}{\phi_{\text{\tiny QO}}}
\newcommand{\phigcvone}{\phi_{\text{\tiny GCV1}}}
\newcommand{\phigcvtwo}{\phi_{\text{\tiny GCV2}}}

\newcommand{\psidp}{\psi_{\text{\tiny DP}}}
\newcommand{\R}{\mathbb{R}}
\newcommand{\RRE}{\text{RRE}}

\newcommand{\ttrace}{\text{trace}}

\newcommand{\hC}{\hat{B}}
\newcommand{\hbC}{\bar{\hat{B}}}
\newcommand{\bC}{\bar{B}}
\newcommand{\tol}{\tau}

\usepackage{graphicx}
\usepackage{xcolor}
\usepackage{mathtools}

\definecolor{malena}{rgb}{0.62, 0.0, 0.77}
\newcommand{\sg}{\textcolor{black}}
\newcommand{\msl}{\textcolor{black}}
\newcommand{\sgnew}{\textcolor{black}}
\newcommand{\mslnew}{\textcolor{black}}

\providecommand{\keywords}[1]{\textbf{\textit{Keywords---}} #1}

\title{Adaptive Regularization Parameter Choice Rules\\for Large-Scale Problems}
\author{Silvia Gazzola\footremember{alley}{Department of Mathematical Sciences. University of Bath, UK.\newline 
Email: \texttt{\{S.Gazzola,M.Sabate.Landman\}@bath.ac.uk}} and Malena Sabat\'e Landman\footrecall{alley}}

\begin{document}

\maketitle

\normalsize

\begin{abstract}
\noindent This paper derives a new class of adaptive regularization parameter choice strategies that can be effectively and efficiently applied when regularizing large-scale linear inverse problems by combining standard Tikhonov regularization and projection onto Krylov subspaces of increasing dimension (computed by the Golub-Kahan bidiagonalization algorithm). The success of this regularization approach heavily depends on the accurate tuning of two parameters (namely, the Tikhonov parameter and the dimension of the projection subspace): these are simultaneously set using new strategies that can be regarded as special instances of bilevel optimization methods, which are solved by using a new paradigm that interlaces the iterations performed to project the Tikhonov problem (lower-level problem) with those performed to apply a given parameter choice rule (higher-level problem). The discrepancy principle, the GCV, 
the quasi-optimality criterion, and Regi\'{n}ska criterion can all be adapted to work in this framework. 
The links between Gauss quadrature and Golub-Kahan bidiagonalization are exploited to prove convergence results for the discrepancy principle, and to give insight into the behavior of the other considered regularization parameter choice rules. 
Several numerical tests modeling inverse problems in imaging show that the new parameter choice strategies lead to regularization methods that are reliable, and intrinsically simpler and cheaper than other strategies already available in the literature.
\end{abstract}


\section{Introduction}\label{sec:intro}
This paper considers linear, large-scale, discrete ill-posed problems of the form
\begin{equation}\label{eq:linsys}
Ax + e = b\,,
\end{equation}
where the matrix $A \in \mathbb{R}^{m \times n}$ is ill-conditioned with ill-determined rank (i.e., the singular values of $A$ quickly decay and cluster at zero without an evident gap between two consecutive ones), $x\in\R^n$ is the desired solution, $b\in\R^m$ is the available right-hand side vector, which is affected by some unknown Gaussian white noise $e\in\R^m$. Problems like this model inverse problems arising in a variety of applications, which typically stem from the discretization of first-kind Fredholm integral equations (see \cite{PCH10} and the references therein). 

It is well-known that, in order to compute a good approximation to $x$, one should regularize (\ref{eq:linsys}), i.e., replace (\ref{eq:linsys}) with a problem closely related to it that is less sensitive to perturbations in the data. Although many approaches are possible to achieve this, in this paper we focus on the standard Tikhonov regularization method, which consists in computing
\begin{equation}\label{tikh}
x({\alpha}) = \arg\min_{x\in\R^n}\:\underbrace{{\|Ax-b\|^2}
\,+\,\alpha{\|x\|^2}}
_{=:F(x,\alpha)}\:,
\end{equation}
where the regularization parameter $\alpha\geq0$ has a key role in balancing the effect of the fit-to-data term $\|Ax-b\|^2$ and the regularization term $\|x\|^2$; one typically assumes that  $\alpha$ is between the smallest and the largest singular value of $A$. 
Here and in the following we use the shorthand notation $\|\cdot\|=\|\cdot\|_2=\langle\cdot,\cdot\rangle^{1/2}$ to denote the vector 2-norm, induced by the standard inner product $\langle\cdot,\cdot\rangle$ on $\R^d$, $d\geq 1$. Many 
parameter choice rules have been already derived to choose a suitable $\alpha$ in (\ref{tikh}): the ones 
considered in this paper can be expressed in the framework
\begin{equation}\label{GenP}
\min_{\alpha\geq 0} P(x(\alpha)) \quad\mbox{subject to}\quad x({\alpha}) = \arg\min_{x\in\R^n}F(x,\alpha)\,, 
\end{equation}
where $P(x(\alpha))$ is a condition to allows a suitable choice of $\alpha\geq 0$, and $F(x,\alpha)$ is defined in (\ref{tikh}). A list of the functionals $P(x(\alpha))$ considered in this paper is given in Table \ref{tab:rules} (some expressions, like the discrepancy principle one, may seem unconventional and will be explained in Section \ref{sec:new}). 
%
\newcolumntype{C}[1]{%
 >{\vbox to 6ex\bgroup\vfill\centering}%
p{#1}%
 <{\egroup}}  
 \begin{table}
\centering
\caption{Functionals $P(x(\alpha))$ in (\ref{GenP}) associated to the parameter rules considered in this paper.}\label{tab:rules}
\footnotesize
\begin{tabular}{C{9.4cm}|C{4.2cm}}
\hline
\multicolumn{1}{c|}{$P(x(\alpha))$} & \multicolumn{1}{c}{\textbf{Parameter rules}} \tabularnewline \hline
 \begin{equation}\label{discrPtab}
 b^T\left((\alpha I - \alpha\heps^2 I + 2AA^T\log((AA^T+\alpha I)^{-1})\right)b - b^T(AA^T)^2x(\alpha) \tag{DP}
 \end{equation} & discrepancy principle, \cite{discrP}\\
(where $\heps \simeq \|e\|/\|b\|$, see (\ref{eq:linsys}))\tabularnewline \hline
\begin{equation}\label{GCVtab}
\frac{\|b-Ax(\alpha)\|^2}{(\ttrace(I-A(A^TA+\alpha I)^{-1}A^T))^2}\tag{GCV}
\end{equation} & generalized cross validation, \cite{GCVbib}\\(GCV)
\tabularnewline \hline
\begin{equation}\label{qotab}
\alpha^2x(\alpha)^T(A^TA+\alpha I)^{-2}x(\alpha)\tag{QO}
\end{equation} & quasi-optimality criterion, \cite{QObib}\tabularnewline \hline 
\begin{equation}\label{Reginskatab}\tag{R}
\|b-Ax(\alpha)\| \|x(\alpha)\|
\end{equation} & Regi\'{n}ska criterion, \cite{Regins} \tabularnewline \hline
\end{tabular}
\end{table}
Problem (\ref{GenP}) is formally a bi-level optimization problem, consisting of a lower-level optimization problem whose solution $x(\alpha)$ is an argument of the higher-level minimization problem; see \cite{bilevelPock}. Thanks to the particular form of $F(x,\alpha)$, one can derive a closed-form solution for $x(\alpha)$, and substitute its expression in $P(x(\alpha))$, so that problem (\ref{GenP}) is essentially a single-level optimization problem. However, in practice, one can obtain $x(\alpha)$ directly only when some factorizations of $A$ (such as the SVD) can be computed: this is not the case for large-scale unstructured problems (\ref{eq:linsys}). In these situations, one should resort to an iterative linear solver to approximate the solution $x(\alpha)$ of the lower-level problem in (\ref{GenP}), while a nonlinear solver is used to compute an approximation to the higher-level problem in (\ref{GenP}).  
Because of this, problem (\ref{GenP}) should still be treated as a bi-level optimization problem. 
In particular, an inner-outer iteration scheme is naturally established when solving (\ref{GenP}), which involves 
two stopping criteria: one for the inner iterations (to be repeatedly applied), and one for the outer iterations. This paper considers only Krylov projection methods based on Golub-Kahan bidiagonalization (GKB) to compute an approximation to the lower-level problem where, at iteration $k$, given a Krylov solution  subspace $\Kry_k$ of dimension $k$, an approximation $x_k(\alpha)\in\Kry_k$ of the solution $x(\alpha)$ of (\ref{tikh}) is computed imposing some additional constraints. 
%
%
%
%

When solving problem (\ref{GenP}), two classical approaches 
are possible. The first obvious (but computationally expensive) one is to repeatedly solve problems of the form (\ref{tikh}), once for every value of the regularization parameter computed within the iterations of $\min_{\alpha\geq 0}P(x(\alpha))$. This approach is equivalent to applying the well-known variable projection method to (\ref{GenP}) (see, e.g., \cite{VPM}), and it is sketched in Algorithm \ref{alg:A1}.
\begin{algorithm}
\caption{Variable projection method for problem (\ref{GenP}).}
\label{alg:A1}
\begin{algorithmic}[1]
\State Choose an initial guess $\alpha_0$.
\For{$\ell=1,2,\dots$ until a stopping criterion is satisfied}
\For{$k=1,2,\dots$ until a stopping criterion is satisfied}
\State{Apply an iterative method to (\ref{tikh}) to approximate $x_k(\alpha_\ell)=x_{k(\ell)}(\alpha_\ell)$.}
\EndFor
\State{Apply a step of a nonlinear solver to compute $\alpha_{\ell+1}$ (given $\alpha_\ell$ and $x_k(\alpha_\ell)$}
\EndFor
\State{Take $x_k(\alpha_\ell)=x_{k(\ell)}(\alpha_\ell)$ as an approximation of the solution of (\ref{tikh}).}
\end{algorithmic}
\end{algorithm}
Note that the intermediate values of $\alpha$ so determined are always computed for the full-dimensional problem (\ref{tikh}), which is therefore solved multiple times. In other words, Algorithm \ref{alg:A1} corresponds to a ``first regularize then project'' approach (see \cite[\S 6.4]{PCH10}). As highlighted by the notation $x_{k(\ell)}(\alpha_{\ell})$, the number of iterations in lines 3--5 of Algorithm \ref{alg:A1} depends in general on the current value of $\alpha_{\ell}$. 
At step 4 of Algorithm \ref{alg:A1}, one can apply any iterative solver for linear least squares problems; 
in particular, when using a Krylov method, most of the computations performed to determine $x_k(\alpha_\ell)=x_{k(\ell)}(\alpha_\ell)$ at the $\ell$th outer iteration can be smartly rearranged or recycled to compute $x_k(\alpha_{\ell+1})=x_{k(\ell+1)}(\alpha_{\ell+1})$ at the $(\ell+1)$th outer iteration; see, for instance, \cite{BGV,fastCG,BolOpt}. 

The second and usually 
more computationally convenient approach is to perform a so-called hybrid method, which projects problem (\ref{tikh}) onto Krylov subspaces $\Kry_k$ of increasing dimension $k$. 
The main claimed advantage of hybrid methods is that, if $k\ll\min\{n,m\}$, one can compute a good regularization parameter for small-scale projected problems only; 
see \cite{ChungOpApprox, chooseRP}. 
Indeed, when employing a hybrid method, instead of (\ref{GenP}) one solves a sequence of bi-level optimization problems of the form
\begin{equation}\label{GenPproj}
\min_{\alpha_k\geq 0} P_k(x_k(\alpha_k)) \quad\mbox{subject to}\quad x_k({\alpha_k}) = \arg\min_{x_k\in\Kry_k}F(x_k,\alpha_k)\,, 
\end{equation}
where the functional $P_k$ appearing in the higher-level problem is a specific regularization parameter choice rule to be employed at the $k$th iteration (it is often a projected version of the strategies listed in Table \ref{tab:rules}), and the lower-level problem of order $k$ is the projection of problem (\ref{tikh}) onto the space $\Kry_k$ 
(see Section \ref{sec:Hybrid} for more details). A common framework for hybrid methods is sketched in Algorithm \ref{alg:A2}.
\begin{algorithm}
\caption{Hybrid method for problem (\ref{GenP}).}
\label{alg:A2}
\begin{algorithmic}[1]
\For{$k=1,2,\dots$ until a stopping criterion is satisfied}
\State{Compute the Krylov subspace $\Kry_k$ and project problem (\ref{tikh}).}
\For{$\ell=1,2,\dots$ until a stopping criterion is satisfied}
\State{Apply a nonlinear solver to approximate $\alpha_k=\alpha_{k(\ell)}$ and $x_k(\alpha_k)=x_k(\alpha_{k(\ell)})$ in (\ref{GenPproj}).} 
\EndFor
\EndFor
\State{Take $x_k(\alpha_k)=x_k(\alpha_{k(\ell)})$ as an approximation of the solution of (\ref{tikh}).}
\end{algorithmic}
\end{algorithm}
%
Note that, when solving (\ref{GenPproj}), one fully runs (till convergence) a parameter choice strategy, with the outcome of selecting a suitable regularization parameter for iteration $k$ (i.e., this is in principle a local choice, good for the $k$th projected Tikhonov problem only). 
%
%
%
%
Indeed, when performing hybrid methods, it is often observed that the regularization parameter that is good for the projected  problem may not be good for the full-dimensional problem \cite{chooseRP}: therefore, solving (\ref{GenPproj}) to high precision for all $k$'s may be worthless;  
nonetheless, when $k$ increases, the regularization parameter obtained applying (\ref{GenPproj}) seems to stabilize around a value that is good for the full-dimensional problem, too; see \cite{ChungOpApprox}. As highlighted by the notation $\alpha_{k(\ell)}$, 
the value of the regularization parameter to be employed for the order $k$ lower-level problem in (\ref{GenPproj}) depends on the number of iterations in lines 3--5 of Algorithm \ref{alg:A2}; however, to keep lighter notations, in the following only $\alpha_k$ will be used. 
We can regard hybrid methods as two-parameter methods, where regularization is achieved by jointly and carefully tuning both $k$ and $\alpha_k$; in general the optimal regularization parameter $\alpha_k$ (i.e., the one minimizing the error) increases with $k$, as the projected problem becomes increasingly ill-conditioned and needs more regularization. Since the projected Tikhonov problem is of order $k$, if $k\ll n$ the lower-level problem in (\ref{GenPproj}) can be solved directly, and (\ref{GenPproj}) is indeed a single-level optimization problem: in this setting, only a stopping criterion for the higher-level problem should be set; however, one needs an additional (and often heuristic) stopping criterion to set $k$, 
i.e., to guarantee that problem (\ref{GenPproj}) is a good approximation to problem (\ref{GenP}). According to the classification in \cite[\S 6.4]{PCH10}, hybrid methods are ``first project then regularize'' approaches. There is a rich literature on parameter choice rules adapted or specific for hybrid methods; see, for instance, \cite{Calvetti2004,Chung08,Fenu16,GNR15,Iveta,RenautHybrid} and the references therein.

The goal of this paper is to introduce a new efficient class of parameter choice strategies for large-scale problems (\ref{tikh}), which leverage ideas typical of the hybrid approach to (\ref{GenPproj}), but are applied directly to (\ref{GenP}). In particular, by an innovative use of projection methods (i.e., Krylov subspace methods based on the GKB algorithm) and by exploiting their connections to Gaussian quadrature rules, 
the new strategies simultaneously compute a value for $k$, {$\alpha_k$ and $x_{k}(\alpha_k)$}, 
thereby computing a good approximation of the solution of the original problem (\ref{GenP}). 
The core idea behind the new strategis is to ``interlace'' 
the iterations needed to solve the lower-level problem and the higher-level problem in (\ref{GenP}). These strategies result in 
only one iteration cycle, bypassing both the approaches in Algorithms \ref{alg:A1} and \ref{alg:A2}. Namely (as sketched in Algorithm \ref{alg:new}), each iteration of the new methods consists in 
performing one step of a projection method for solving the linear lower-level problem in (\ref{GenP}), 
and one step of an iterative scheme for solving the nonlinear higher-level problem in (\ref{GenP}). 
\begin{algorithm}
\caption{New adaptive algorithm for problem (\ref{GenP}).}
\label{alg:new}
\begin{algorithmic}[1]
\State Choose an initial guess $\alpha_1$.
\For{$k=1,2,\dots$ until a stopping criterion is satisfied}
\State{Compute the Krylov subspace $\Kry_k$ and project problem (\ref{tikh}).}
\State{Apply a step of a nonlinear solver to compute $\alpha_{k+1}$ (given $\Kry_k$ and $\alpha_k$).}
\EndFor
\State{Take $x_k(\alpha_{k+1})\in\Kry_k$ as an approximation of the solution of (\ref{tikh}).}
\end{algorithmic}
\end{algorithm}
Note that, when performing Algorithm \ref{alg:new}, the approximation subspace for the solution of (\ref{GenP}) is enlarged while a suitable value for $\alpha$ is set. As already mentioned, Algorithm \ref{alg:new} avoids nested iteration cycles, so that only one stopping criterion should be set (this is typically a standard stopping criterion applied to the higher-level problem in (\ref{GenP})). The new strategy, in addition to being conceptually simpler, potentially allows for great computational savings: this is obvious when compared to the approach in Algorithm \ref{alg:A1}; however, note that, for each $k$, the approach in Algorithm \ref{alg:A2} still requires the repeated solution of the lower-level problem (\ref{GenPproj}) which may become expensive when $k$ increases. When the functional $P(x(\alpha))$ in (\ref{GenP}) is the discrepancy principle, 
convergence of the couple $(x_k,\alpha_k)$ computed by Algorithm \ref{alg:new} to the solution $(x,\alpha)$ of (\ref{GenP}) can be proven. Fort the other functionals listed in Table \ref{tab:rules}, theoretical insight into the behavior of Algorithm \ref{alg:new} can be provided. 

We must mention that an approach similar to the adaptive strategies presented in this paper was already derived in \cite{GN} (the so-called ``secant update method''). However, the secant update method handles the discrepancy principle only, and no formal convergence proof was provided in \cite{GN}. The present paper still considers the discrepancy principle as a possible choice for the functional $P(x(\alpha))$ in (\ref{GenP}), but adopts a different nonlinear solver with respect to the secant update method, and gives a convergence proof for the new strategy (when GKB is used to project the linear lower-level problem in (\ref{GenP})). Moreover, the present paper extends this framework to handle all the choices of $P(x(\alpha))$ listed in Table \ref{tab:rules}. We also remark that the idea of exploiting the links between GKB and Gaussian quadrature rules to choose the regularization parameter in (\ref{tikh}) is not completely new: for instance, the authors of \cite{VonMatt} adopt Gaussian quadrature rules to estimate a value of $\alpha$ in the full-dimensional problem (\ref{tikh}) according to GCV, and the authors of \cite{Calvetti1999, Calvetti2004, Fenu16} explore a variety of parameter choice methods (including some of the ones listed in Table \ref{tab:rules}) to be employed in Algorithm \ref{alg:A2}, using Gaussian quadrature rules to link some projected functionals $P_k(x_k(\alpha_k))$ to their full-dimensional counterparts $P(x(\alpha))$, and to set stopping criteria for the number of iterations $k$. The approach proposed in this paper is novel in that GKB and Gaussian quadrature rules are employed in the framework of bi-level optimization problems, and values of $\alpha_k$, and $x_k(\alpha_k)$ approximating the solution of (\ref{GenP}) are simultaneously computed within only one iteration cycle. 


The remaining part of this paper is organized as follows. Section \ref{sec:background} recalls some background material. 
Section \ref{sec:new} unfolds the theory and implementation of the new class of adaptive parameter choice methods. Section \ref{sec:numerics} presents some numerical experiments and comparisons. Section \ref{sec:end} presents concluding remarks.  

\section{Background}\label{sec:background} 

This section briefly recalls basic facts about regularizing Krylov methods based on Golub-Kahan bigiagonalization (GKB), which are the backbones of the strategies proposed in this paper for the solution of (\ref{GenP}), and which are more carefully detailed in \cite[Chapter 4]{NumMethMatrix} and \cite[Chapter 6]{PCH10}. Also some specific links between GKB and Gauss quadrature are briefly recalled (a more general and complete description can be found in \cite{MMQ}): these will be used to derive approximations for the functionals in Table \ref{tab:rules} 
and for devising convergence proofs. 

\subsection{GKB-based iterative regularization methods}\label{sec:Hybrid}
%

Given a matrix $A\in\R^{m\times n}$ and a vector $b\in\R^m$, the $k$th iteration of the GKB algorithm consists in updating partial matrix factorizations of the form
\begin{eqnarray} \label{eq:Lanczos1}
A V_k = U_k B_k + \sigma_{k+1} u_{k+1}e^T_k=U_{k+1} \bB_k \,, \qquad A^T U_k = V_k B_k^T\,,
\end{eqnarray} 
where $V_k \in \mathbb{R}^{n \times k}$ and $U_{k+1}=[U_k,u_{k+1}]=[u_1,\dots,u_k,u_{k+1}] \in \mathbb{R}^{m \times (k+1)}$, with 
$u_1=b/\|b \|$, 
are matrices whose orthonormal columns span the Krylov subspaces $\mathcal{K}_{k}(A^{T} A,A^{T} b)$ and $\mathcal{K}_{k}(A A^{T}, b)$, respectively; 
$B_k$ and $\bB_k$ are lower bidiagonal matrices of the form
\begin{equation}\label{eq:Ck}
B_k=\left[ \begin{array}{ccccc}
\rho_1 &  &  & &\\
\sigma_2 &  \rho_2 & && \\
&\ddots&\ddots&& \\
& & \sigma_{k-1}& \rho_{k-1} & \\\
&&&\sigma_{k} & \rho_k \end{array} \right] \in \R^{k \times k},\quad
\bB_k=\left[ \begin{array}{c}
B_k\\
\sigma_{k+1}e_k^T
\end{array} \right] \in \R^{(k+1) \times k}.
\end{equation}
Here and in the following, $e_i$ denotes the $i$th canonical basis vector of $\R^d$, $d\geq i$. The following 
\begin{equation}\label{nobreak}
\mbox{\textbf{assumption}: the GKB algorithm (\ref{eq:Lanczos1}) does not breakdown,} 
\end{equation}
i.e., $\rho_k,\,\sigma_{k}>0$ for all $k\leq \min\{m,n\}$, will be made through the paper.

It is well-known that many Krylov methods based on GKB are iterative regularization methods, with the number of iterations acting as a regularization parameter. One of the most widespread methods in this class is  arguably LSQR, which is mathematically equivalent to CGLS. The $k$th LSQR iteration approximates the solution of (\ref{eq:linsys}) by taking
\[
x_k=V_ky_k,\quad\mbox{where}\quad y_k=\arg\min_{y\in\R^k}\|\bB_k y-\|b\|e_1\|\,.
\]
By exploiting the first decomposition in (\ref{eq:Lanczos1}) and the properties of the matrices appearing therein, one can easily see that the LSQR solution minimizes the norm of the residual $r_k=b-Ax_k$ among all the vectors belonging to the space $\Kry_k(A^TA,A^Tb)$. 
%

As already hinted in Section \ref{sec:intro}, Krylov methods based on GKB are also commonly employed as hybrid regularization methods (Algorithm \ref{alg:A2}): at the $k$th iteration of a GKB-based hybrid method, the Tikhonov problem (\ref{tikh}) is projected onto the space $\Kry_k(A^TA,A^Tb)$, obtaining
\begin{equation}\label{hybridprojpb}
x_k(\alpha_k) = V_k y_k(\alpha_k)\,, \quad\mbox{where} \quad  y_k(\alpha_k)=  
\arg\min_{y\in\R^k}\: \|\bB_k y-\|b\|e_1\|^2+\alpha_k \|y\|^2\,,
\end{equation}
where the first decomposition in (\ref{eq:Lanczos1}) and the properties of the matrices appearing therein have been used; the iteration-dependent regularization parameter $\alpha_k$ can be determined by solving (\ref{GenPproj}) (lines 3--5 of Algorithm \ref{alg:A2}). The claimed main upside of hybrid methods is their reduced sensitivity to the stopping criterion for the iterations $k$, which allows to compute a typically more accurate solution in larger Krylov subspaces $\Kry_k$ with respect to purely iterative methods; see, for instance, \cite{ChungOpApprox,OS}.

%

The symmetric Lanczos and the GKB algorithms are closely related, in that, multiplying the second equation in (\ref{eq:Lanczos1}) from the left by $A$, and using the first equation in (\ref{eq:Lanczos1}), one obtains
\begin{eqnarray}\label{SymLanczos1}
A A^T U_k &= A V_k B_k  =U_k \underbrace{B_k B_k^T}_{=:T_k} + \sigma_{k+1}  u_{k+1}e^T_k B_k^T = 
U_k T_k + \sigma_{k+1} \rho_{k} u_{k+1}e^T_k\,.
\end{eqnarray} 
Here, the lower bidiagonal matrix $B_k$ defined in (\ref{eq:Ck}) can also be regarded as the Cholesky factor of the symmetric positive definite tridiagonal matrix $T_k=B_kB_k^T$ obtained after $k$ iterations of the symmetric Lanczos algorithm applied to $A A^T$ with initial vector $b$. 
Moreover, multiplying the first expression in (\ref{eq:Lanczos1}) from the left with $A^T$, and using again the second equation in (\ref{eq:Lanczos1}), one obtains
\begin{eqnarray}
A^T A V_k &=& A^T U_k B_k  + \sigma_{k+1} A^T  u_{k+1}e^T_k = A^T U_{k+1} \bar{B}_{k} = V_{k+1} B_{k+1}^T \bar{B}_{k}\nonumber\\ 
&=& V_{k} \underbrace{\bar{B}_{k}^T \bar{B}_{k}}_{=:\hat{T}_k}+\rho_{k+1} \sigma_{k+1} v_{k+1}e^T_{k},\label{SymLanczos2}
\end{eqnarray}
so that $V_k$ can be regarded as the matrix generated by performing $k$ steps of the symmetric Lanczos algorithm applied to $A^T A$, with initial vector $A^T b$. After computing the QR-factorization $\bar{B}_k= Q_k \hat{B}_k^T$, where $\hat{B}_k\in\R^{k\times k}$ is lower bidiagonal, one can see that $\hat{B}_k^T$ is the Cholesky factor of the symmetric positive definite tridiagonal matrix $\hat{T}_k=\bar{B}_k^T\bar{B}_k=\hat{B}_k \hat{B}^T_k$. 

\subsection{Gauss quadrature for approximating \msl{quadratic} forms}\label{sec:Gauss}

Let $C \in \mathbb{R}^{p \times p}$ be a symmetric semi-positive definite matrix, having spectral decomposition \linebreak[4]$C = W  \Lambda W^T$, where $ \Lambda$ is a diagonal matrix whose diagonal elements are the eigenvalues $0\leq\lambda_1\leq\lambda_2\leq\dots\leq\lambda_p$ of $C$, and $W$ is the orthonormal matrix whose columns are the normalized eigenvectors of $C$. This section presents a strategy to compute approximations or bounds for general \msl{quadratic} forms
\begin{equation} \label{eqn:generallinearform}
    f(\phi,C,u)=u^T \phi(C) u\,,
\end{equation}
where $u\in\R^p$ is a given vector and $\phi$ is a given smooth function on the interval $[0,+\infty)$ 
of the real line. Using  standard definitions and derivations, (\ref{eqn:generallinearform}) can be expressed as 
\begin{equation}\label{quadratic}
f(\phi,C,u)=u^T \phi(C) u = u^T W \phi(\Lambda) W^T u = 
\sum^p_{i=1} \phi(\lambda_i) (W^T u)^2_i=
\int_{0}^{+\infty}\!\!\! \phi(\lambda) d\omega(\lambda)=:I(\phi)\,.
\end{equation}
The last equality comes from considering the sum as 
a Riemann-Stieltjes integral, 
where the distribution function $\omega$ is a \sg{non-decreasing} 
step function with jump discontinuities at the eigenvalues $\lambda_i$. The chain of equalities (\ref{quadratic}) makes it natural to consider quadrature rules to approximate the quadratic form in (\ref{eqn:generallinearform}). Gaussian quadrature rules will be employed for this purpose, and they will be computed applying the symmetric Lanczos algorithm to $C$ with initial vector $u$. 
%
%
In the following sections, only particular instances of (\ref{eqn:generallinearform}) are taken into account, which appear in the definition of the functionals $P(x(\alpha))$ listed in Table \ref{tab:rules}. Indeed, only quadratic forms of the kind $b^T\phi(AA^T)b$ and $(A^Tb)^T\phi(A^TA)(A^Tb)$ have to be approximated, so that only the symmetric Lanczos algorithm applied to $AA^T\in\R^{m\times m}$ with initial vector $b\in\R^{m}$, or applied to $A^TA\in\R^{n\times n}$ with initial vector $A^Tb\in\R^{n}$, have to be considered: this is done implicitly by applying the breakdown-free GKB algorithm to $A\in\R^{m\times n}$ and $b\in\R^{m}$ (see assumption (\ref{nobreak}) and equations (\ref{SymLanczos1}) and (\ref{SymLanczos2})). 

Let $T_k=B_kB_k^T\in\R^{k\times k}$ be the symmetric positive definite tridiagonal matrix appearing in (\ref{SymLanczos1}), 
produced after performing $k\leq\min\{n,m\}$ steps of the Lanczos algorithm 
applied to the matrix $A A^T$ with initial vector $b$. Let $\{q_i(\lambda)\}_{i=0}^{\sg{k}}$ be the family of orthonormal polynomials with respect to the inner product induced by the measure $\omega(\lambda)$ (associated to $A A^T$ and $b$), and let $T_k= Y_k \Theta_k Y_k^T$ be the spectral decomposition of $T_k$, where $Y_k$ is the orthonormal matrix whose columns are the normalized eigenvectors of $T_k$, and $\Theta_k$ is the diagonal matrix of eigenvalues $0<\theta_1\leq\dots\leq\theta_k$. It is well-known that the $k$-point Gauss quadrature rule with respect to the measure $\omega(\lambda)$, defined as
\begin{equation}\label{GQ1}
\mathcal{G}_k(\phi ,A A^T, b):= \sum_{j=1}^{k} \phi(\theta_j) \underbrace{||b||^2 (e_1^T(Y_k)e_j)^2}_{=\mu_j}
=  ||b||^2 e_1^T Y_k \phi(\Theta_k) Y_k^T e_1=||b||^2 e_1^T \phi(T_k) e_1\,,
\end{equation}
approximates (\ref{quadratic}) with $C=AA^T$ and $u=b$. More specifically, the 
eigenvalues of $T_k$ are the zeros of the polynomial \sg{$q_k(\lambda)$}, as well as the quadrature nodes, and 
the quadrature weights $\mu_j$ are given by the \sg{rescaled and squared} first components of the eigenvectors of $T_k$. 
Analogously, the $k$-point Gauss-Radau quadrature rule with one assigned node at the origin and with respect to the measure $\omega(\lambda)$, approximating 
(\ref{quadratic}) with $C=AA^T$ and $u=b$, can be obtained by suitably modifying the symmetric Lanczos process to compute a symmetric semi-positive definite matrix $\bar{T}_k$ of order $k$ with one prescribed eigenvalue at the origin. This amounts to taking $\bar{T}_k=\bar{B}_{k-1}\bar{B}_{k-1}^T\in\R^{k\times k}$, where $\bar{B}_{k-1}$ is the $(k-1)\times k$ version of the matrix $\bar{B}_k$ in (\ref{eq:Ck}) (or, alternatively, is the matrix obtained by selecting the first $(k-1)$ columns of the Cholesky factor $B_k$ of $T_k$); see \cite{VonMatt} for a proof. 
Eventually, such a quadrature rule reads
\begin{equation}\label{GRQ1}
\mathcal{R}_k(\phi ,A A^T, b):= \sum_{j=1}^{k} \phi(\bar{\theta}_j)||b||^2 (e_1^T(\bar{Y}_k)e_j)^2
=  ||b||^2 e_1^T \bar{Y}_k \phi(\bar{\Theta}_k) \bar{Y}_k^T e_1=||b||^2 e_1^T \phi(\bar{T}_k) e_1\,,
\end{equation}
where $\bar{T}_k= \bar{Y}_k \bar{\Theta}_k \bar{Y}_k^T$ is the spectral decomposition of $\bar{T}_k$, with $\bar{Y}_k$ orthonormal and \linebreak[4]$\bar{\Theta}_k=\text{diag}(\bar{\theta}_1,\dots,\bar{\theta}_k)$, $0=\bar{\theta}_1<\bar{\theta}_2\leq\dots\leq\bar{\theta}_k$.
%

Now let $\hat{T}_k=\hat{B}_k\hat{B}_k^T \in \R^{k \times k}$ be the symmetric positive definite tridiagonal matrix appearing in (\ref{SymLanczos2}), produced after performing $k\leq \min\{n,m\}$ steps of the Lanczos algorithm  applied to $A^T A$ with initial vector $A^T b$. Similarly to the previous derivations, the $k$-point Gauss quadrature rule with respect to the measure $\omega(\lambda)$ (associated to $A^TA$ and $A^Tb$), defined as
\begin{equation}\label{GQ2}
\mathcal{G}_k(\phi,A^T A, A^T b):= \sum_{j=1}^{k} \phi(\hat{\theta}_j)||A^Tb||^2 (e_1^T(\hat{Y}_k)e_j)^2
=||A^Tb||^2 e_1^T \phi(\hat{T}_k) e_1\,,
\end{equation}
approximates (\ref{quadratic}) with $C=A^TA$ and $u=A^Tb$. Here, using notations analogous to the previous ones, $\hat{T}_k= \hat{Y}_k \hat{\Theta}_k \hat{Y}_k^T=\hat{Y}_k\text{diag}(\hat{\theta}_1,\dots,\hat{\theta}_k)\hat{Y}_k^T$ is the spectral decomposition of the matrix $\hat{T}_k$. 
Finally, the matrix $\bar{\hat{T}}_k := \bar{\hat{B}}_{k-1} \bar{\hat{B}}_{k-1}^T\in\R^{k\times k}$, where $\bar{\hat{B}}_{k-1}$ is the matrix constructed by selecting the first $k-1$ columns of the Cholesky factor $\hat{B}_{k}$ of $\hat{T}_k$, is symmetric semi-positive definite with one prescribed eigenvalue at the origin. Therefore, the $k$-point Gauss-Radau quadrature rule with respect to the measure $\omega (\lambda)$ (associated to $A^T A$ and $A^T b$), with one node assigned at the origin, can be expressed as 
\begin{equation}\label{GRQ2}
\mathcal{R}_k(\phi, A^T A, A^T b):= \sum_{j=1}^{k} \phi(\bar{\hat{\theta}}_j)||A^Tb||^2 (e_1^T(\bar{\hat{Y}}_k)e_j)^2
=||A^Tb||^2 e_1^T \phi(\bar{\hat{T}}_k) e_1\,,
\end{equation}
where $\bar{\hat{T}}_k= \bar{\hat{Y}}_k \bar{\hat{\Theta}}_k \bar{\hat{Y}}_k^T=\bar{\hat{Y}}_k\text{diag}(\bar{\hat{\theta}}_1,\dots,\bar{\hat{\theta}}_k)\bar{\hat{Y}}_k^T$ is the spectral decomposition of the matrix $\bar{\hat{T}}_k$. 


Assuming that $\phi$ is a $2k$-times differentiable function, the quadrature errors $\mcE_{\mcQ_k}(\phi)=I(\phi)-\mcQ_k(\phi)$ associated to $k$-point Gauss and Gauss-Radau quadrature rules (with $\mcQ_k(\phi)=\mcG_k(\phi)$ and $\mcQ_k(\phi)=\mcR_k(\phi)$, respectively, and where the dependence on the matrices $AA^T$, $A^TA$, and the vectors $b$, $A^Tb$, has been removed in the interest of generality), are given by
\begin{equation}\label{GError}
\mcE_{\mcG_k}(\phi)=\frac{\phi^{(2k)}(\zeta_{\mcG_k})}{(2k)!}\int_{0}^{+\infty}\prod_{i=1}^k(t-\zeta_i)^2 d\omega(t)
\end{equation}
and
\begin{equation}\label{GRError}
\mcE_{\mcR_k}(\phi)=\frac{\phi^{(2k-1)}(\bar{\zeta}_{\mcR_k})}{(2k-1)!}\int_{0}^{+\infty}t\prod_{i=2}^k(t-\sg{\bar{\zeta}_i})^2 d\omega(t)\,,
\end{equation}
respectively. Here $\zeta_{\mcG_k},\,\bar{\zeta}_{\mcR_k}\in [\lambda_1,\lambda_{\sg{\min\{n.m\}}}]$. The $\zeta_i$'s denote the nodes of a Gauss quadrature rule (so that $\zeta_i=\theta_i$ for (\ref{GQ1}) and $\zeta_i=\hat{\theta}_i$ for (\ref{GQ2})); the $\bar{\zeta}_i$'s denote the nodes of a Gauss-Radau  quadrature rule (so that $\bar{\zeta}_i=\bar{\theta}_i$ for (\ref{GRQ1}) and $\bar{\zeta}_i=\bar{\hat{\theta}}_i$ for (\ref{GRQ2})). As an immediate consequence of formulas (\ref{GError}) and (\ref{GRError}), if the derivatives of the function $\phi$ have constant sign on $[\lambda_1,\lambda_{\sg{\min\{n.m\}}}]$, then upper or lower bounds for quadratic forms of the kind $b^T \phi(AA^T)b$ and $b^TA\phi(A^TA)A^Tb$ can be obtained by employing Gauss and Gauss-Radau quadrature rules of the form (\ref{GQ1})--(\ref{GRQ2}): this will be explored more carefully in the following sections, for specific functionals.

\section{Adaptive parameter choice strategies}\label{sec:new}

This section explains how the GKB algorithm can be adopted in connection with a Newton-like nonlinear solver to approximate the solution of  the bi-level optimization problem (\ref{GenP}). In the following, the closed-form expression $x(\alpha)=(A^TA+\alpha I)^{-1}A^Tb$ and algebraic manipulations thereof will be often used for deriving analytical expressions of the functionals $P(x(\alpha))$ in Table \ref{tab:rules} as quadratic forms (\ref{eqn:generallinearform}) with $C=A^TA$ and $u=A^Tb$, or with $C=AA^T$ and $u=b$. 

Although the methods in this paper are meaningful for large-scale problems, some numerical illustrations involving a moderate-scale problem  generated thorough MATLAB's toolbox \emph{IR Tools} \cite{IRtools} will be presented in this section to show the typical behavior of the functionals $P(x(\alpha))$ in Table \ref{tab:rules} (which can be easily computed once the SVD of $A$ is available) and of upper and lower bounds thereof. Namely, the following instructions are used 
\begin{equation}\label{illustration}
\text{\texttt{optn.trueImage = 'pattern1'; [A,b,x] = PRblur(64,optn); bn = PRnoise(b);}}
\end{equation}
\normalsize
to generate an image deblurring test problem involving a simple geometric test image of size $64\times 64$ pixels (so that the coefficient matrix $A$ has order $4096$), 
a medium Gaussian blur, and Gaussian white noise with $\|e\|/\|b\|=10^{-2}$.

\subsection{Discrepancy Principle}\label{sec:discr} 
The functional associated to the discrepancy principle (\ref{discrPtab}) in Table \ref{tab:rules} can be naturally regarded as a quadratic form
\[
P(x(\alpha))=  b^T\psidp(AA^T,\alpha)b\,,
\]
where
\begin{equation}\label{primdiscralpha}
\psidp(t,\alpha)=\alpha - \alpha\heps^2 + 2t\log((\alpha + t)^{-1}) - t^2(\alpha + t)^{-1}\,,
\end{equation}
is a function defined for $t\geq 0$, $\alpha> 0$ (see also (\ref{quadratic})), and $\heps= \eps /\|b\|$ is an estimate for the noise level $\| e \| / \|b\|$ in (\ref{eq:linsys}) (possible safety factors are already incorporated in $\heps$).

The first and second derivatives with respect to $\alpha$ of the function $\psidp(t,\alpha)$ in (\ref{primdiscralpha}) read
\begin{equation} \label{eqn:psi}
\phidp(t,\alpha):=\da\psidp(t,\alpha) = \alpha^2(t+\alpha)^{-2} - \heps^2 ,\quad \ddaa\psidp(t,\alpha) = 2\alpha t(t+\alpha)^{-3}\,,
\end{equation}
respectively. Since $\ddaa  \psidp(t,\alpha) \geq 0$ for $\alpha> 0$, then $\ddaa  b^T\psidp(AA^T,\alpha)b \geq 0$ for $\alpha> 0$ (i.e., $b^T \psidp(A A^T,\alpha) b$ is convex as a function of $\alpha$ for $\alpha>0$). Therefore, solving (\ref{GenP}) amounts to solving the nonlinear equation 
\begin{equation}\label{eq:discr}
0=b^T\phidp(AA^T,\alpha)b=\alpha^2b^T(AA^T+\alpha I)^{-2}b-\eps^2=\|b-Ax(\alpha)\|^2-\eps^2\,
\end{equation}
with respect to $\alpha$ (see \cite{Calvetti1999} for complete derivations). 
%
Since the continuous function $b^T\phidp(AA^T,\alpha) b$ is increasing in $\alpha$, there exists a unique zero $\alpha^\ast$ in $(0,\infty) $ provided that 
\begin{equation}\label{conddp}
\lim_{\alpha\rightarrow 0}b^T\phidp(AA^T,\alpha)b=-\eps^2<0\quad \text{and}\quad\lim_{\alpha\rightarrow +\infty}b^T\phidp(AA^T,\alpha)b=\|b\|^2-\eps^2>0\,,
\end{equation}
where the last inequality obviously holds if 
$\eps^2<\|b\|^2$ (this is a reasonable bound for the amount of noise in the data, which will be assumed in the following). 
Equation (\ref{eq:discr}) agrees with the standard discrepancy principle formulation 
and 
one can easily apply a zero-finder (e.g., Newton method) to compute $\alpha^\ast$. 
Since $\phidp(t,\alpha)$ is not convex for $\alpha> 0$, Newton method is not guaranteed to globally converge.  
As proposed in \cite{ZeroFind}, the simple change of variable $\beta = 1/\alpha$ is performed in (\ref{eq:discr}), so that
\begin{equation}\label{eq:discrcvx}
\hphidp(t,\beta) := (\beta t+1)^{-2} - \heps^2\quad\sg{\text{and}\quad \hfdp(\hphidp,AA^T,b,\beta):=b^T\hphidp(AA^T,\beta)b-\eps^2}
\end{equation}
\sg{are} decreasing and convex for $\beta>0$, and a unique zero $\beta^\ast$ exists if conditions analogous to (\ref{conddp}) are satisfied. 
Newton method applied to solve the nonlinear equation (with respect to $\beta$)
\begin{equation}\label{eq:discrmatrbeta}
\sg{0=\hfdp(\hphidp,AA^T,b,\beta)=b^T(\beta AA^T+I)^{-2}b-\eps^2}
\end{equation}
globally converges, and can be easily implemented if the SVD of $A$ is available. Since this is not the case in general for large-scale problems (as remarked in Section \ref{sec:intro}), an alternative solution approach for (\ref{eq:discrmatrbeta}) that fits into the framework of Algorithm \ref{alg:new} is derived. 

\paragraph{A modified Newton method for (\ref{eq:discr}).}

The following result proves the convergence of a specific modification of the classical Newton zero finder, \sg{which can be used in a general setting whenever dealing with a sequence of functions $\{\mcL_k\}_{k\geq1}$ satisfying certain assumptions. This method will be later applied to solve (\ref{eq:discr})}. 

\begin{theorem}\label{them:dpnewton}
Let $f:(0,+\infty)\rightarrow \R$ be a strictly decreasing, convex, differentiable function such that 
$\lim_{x\rightarrow +\infty}f(x)<0$. Let $\{\mcL_k\}_{k\geq1}:(0,+\infty)\rightarrow \R$, be a sequence of strictly decreasing, convex, differentiable, increasing lower bounds for $f$, i.e., 
\begin{equation}\label{bounds}
\mcL_k(x)\leq\mcL_{k+1}(x)\leq f(x)\mbox{ for all }k\geq1,\;x\in(0,\msl{+\infty})\,,
\end{equation}
\sg{such that $\lim_{x\rightarrow 0}L_k(x)>0$ 
for all $k\geq 1$, and $\lim_{k\rightarrow +\infty}L_k(x)=f(x)$} for all $x\in(0,\msl{+\infty})$. Then, given $x_1$ such that $\mcL_1(x_1)\geq0$, the sequence $\{x_k\}_{k\geq1}$ obtained from the recursion
\begin{equation}
x_{k+1}=x_k-\frac{\mcL_k(x_k)}{\mcL_k^\prime(x_k)}
\label{newtonstep2}
\end{equation}
monotonically converges to the root $x^\ast$ of $f$ from the left.
\end{theorem}
\begin{proof} 
\sg{The assumptions assure that the functions $f$ and $L_k$, $k\geq 1$, have exactly one zero in $(0,+\infty)$}. Given $x_k$, $k\geq 1$, define the function
\[
t_k(x)=\mcL_k(x_k)+\mcL_k^\prime(x_k)(x-x_k)\,,
\]
i.e., the tangent line to the graph of $\mcL_k$ at $(x_k,\mcL_k(x_k))$. Relation (\ref{newtonstep2}) is established by imposing $t_k(x_{k+1})=0$ and, together with  the convexity of $\mcL_{k}$ and (\ref{bounds}), leads to
\begin{equation}\label{lowerInequalities}
0=t_{k}(x_{k+1}) \leq \mcL_{k}(x_{k+1})\leq \mcL_{k+1}(x_{k+1}) \leq f(x_{k+1}).     
\end{equation}
Replacing $k$ by $k-1$ in the above relation implies $\mcL_{k}(x_{k})\geq0$ which, together with $\mcL_k^\prime\sg{(x_k)}<0$ \sg{and (\ref{newtonstep2})}, leads to $x_k\leq x_{k+1}$. Moreover, since $f$ is decreasing and $f(x_{k+1})\geq 0$, $x_{k+1}\leq x^\ast$. Therefore the sequence $\{x_{\msl{k}}\}_{k\geq 1}$ is monotonically increasing and bounded above by $x^\ast$. \sg{Taking the limit for $k\rightarrow\infty$ in (\ref{lowerInequalities}) implies $L_{k+1}(x_{k+1})\rightarrow f(x_{k+1})$, so that $x_{k+1}$ converges to $x^\ast$ thanks to the convergence of Newton method.}
\end{proof}

\begin{remark}
\sg{Given a sequence of functions $\{\mcL_k\}_{k\geq1}$,} 
the $k$th iteration of the modified Newton method (\ref{newtonstep2}) consists in performing only one (standard) Newton iteration on the $k$th function $\mcL_{k}$. 
 Figure \ref{fig:modNewton} gives a geometrical illustration \sg{of recursion (\ref{newtonstep2})}. 
\begin{figure}[h]
\center
\includegraphics[width=8cm]{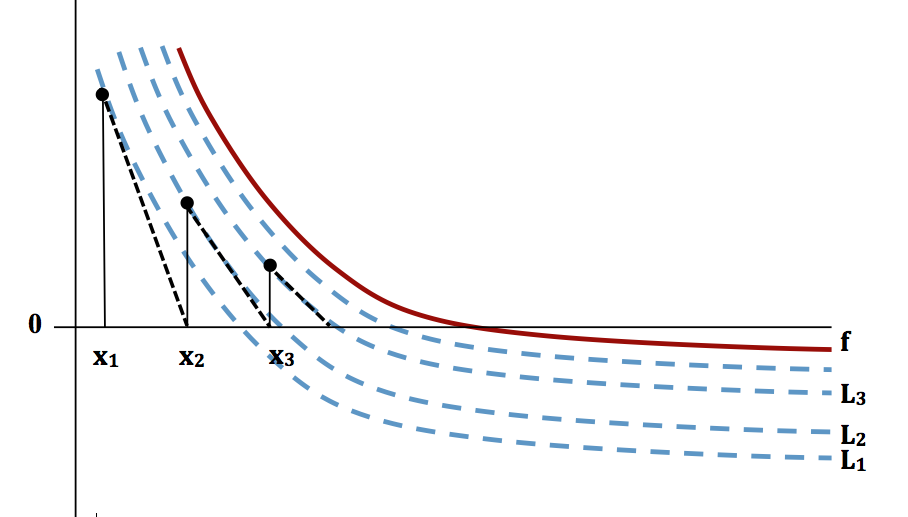}
\caption{Geometrical interpretation of the modified Newton method (\ref{newtonstep2}).}
\label{fig:modNewton}
\end{figure}
\end{remark}
\begin{remark} Theorem \ref{them:dpnewton} still holds if assumption $\mcL_1(x_1)\geq 0$ is removed and $f,\{\mcL_k\}_{\msl{k \geq 1}}:\R\rightarrow\R$ (i.e., considering functions defined on the whole real line). 
Indeed, in this setting
\[
0=t_1(x_2)\leq\mcL_1(x_2),
\]
so that the reasoning in the proof of Theorem \ref{them:dpnewton} can be applied starting from $x_2$. 
\end{remark}

\sg{Turning now to the discrepancy principle (\ref{eq:discrmatrbeta}),} since the matrix functional $\hfdp$ is strictly decreasing, convex, and differentiable with respect to $\beta$, and \msl{$\partial^{(2k)}_t \hphidp (t,\beta) < 0 $ for all $k\geq 1$, $t\geq 0, \beta>0$}, lower bounds for $\hfdp$ 
are obtained by applying the Gauss quadrature rule, leading to
\begin{equation}\label{eq:discrmatrbound}
\mathcal{G}_k(\hphidp, AA^T,b, \beta)=\|b\|^2e_1^T(\beta B_k B_k^T+I)^{-2}e_1\ - \eps^2, \quad k\geq 1;
\end{equation}
see Section \ref{sec:Gauss} and equation (\ref{GError}). \sg{These bounds are \sg{increasing} (see \cite[Theorem 2.1]{Lopez08})}, so that, under assumption (\ref{nobreak}),  
\[
\mathcal{G}_1(\beta)\leq \mathcal{G}_2(\beta)\leq\dots\leq \mathcal{G}_{\sg{p}-1}(\beta)\leq \mathcal{G}_{p}(\beta)=\hfdp(\beta)\,,\quad p=\min\{n,m\}\,;
\]
\sg{the shorthand notation $\mathcal{G}_k(\beta)$ has been used for $\mathcal{G}_k(\hphidp,AA^T,b,\beta)$; similarly, $\hfdp$ is defined in (\ref{eq:discrcvx}) and also depends on $\hphidp$, 
$AA^T$, and $b$.} 
The functions $\mcG_k\sg{(\beta)}$, $1\leq k\leq p$, are strictly decreasing, convex, and differentiable with respect to $\beta$, 
$\lim_{\beta\rightarrow 0}\mathcal{G}_k(\beta)=\|b\|^2\sg{-\eps^2}>0$ \sg{(reasonable bound for the amount of noise, see (\ref{conddp}))}, and 
$\lim_{\beta\rightarrow +\infty}\mathcal{G}_k(\beta)=-\eps^2<0$. \sg{The same limits hold for $\hfdp$. }
\sg{The above derivations assure that} the assumptions of Theorem \ref{them:dpnewton} are satisfied, so that the following result holds. 

\begin{corollary}\label{cor1}
Let $A\in\R^{m\times n}$ and $b\in\R^m$ be as in (\ref{eq:linsys}), and let $\hphidp$ be defined as in (\ref{eq:discrcvx}); consider $\hfdp$ in (\ref{eq:discrcvx}) and $\{\mcG_k\}_{k}$ 
in (\ref{eq:discrmatrbound}) as functions of $\beta>0$. Let $\beta_1>0$ be such that $\mcG_1(\beta_1)\geq 0$. Then the sequence $\{\beta_k\}_{k\geq 1}$ obtained from the recursion 
\begin{equation}
\beta_{k+1}=\beta_k-\frac{\mcG_k(\beta_k)}{\mcG_k^\prime(\beta_k)}
\label{newtonstepdiscr}
\end{equation}
monotonically converges to the root $\beta^\ast$ of (\ref{eq:discrmatrbeta}) from the left. 
\end{corollary}

%
\begin{remark}
Relation (\ref{newtonstepdiscr}) is formally similar to (\ref{newtonstep2}). Notably, since $\mathcal{G}_{k}(\beta)=\hfdp(\beta)$ for $k\geq p=\min\{n,m\}$, relation (\ref{newtonstepdiscr}) reduces to (standard) Newton method when $k\geq p$. However, this never happens in practice, because the bounds $\mathcal{G}_{k}(\beta)$ are observed to quickly approach $\hfdp(\beta)$ and the convergence of (standard) Newton method is quadratic; see also Section \ref{sec:numerics}.
\end{remark}
Referring to the framework of Algorithm \ref{alg:new}, it is clear that a new value of the regularization parameter for problem (\ref{tikh}) is computed at the $k$th iteration using relation (\ref{newtonstepdiscr}) (recall that $\alpha_k=1/\beta_k$). To achieve this, at the $k$th iteration of Algorithm \ref{alg:new}, the Krylov subspace $\Kry_k(AA^T,b)$ is needed to compute $\mcG_k$ in (\ref{eq:discrmatrbound}), so that $k$ iterations of the GKB algorithm should be performed (see Section \ref{sec:Gauss} and equation (\ref{GQ1})). The computational cost of this task is dominated by $O(2kmn)$ floating point operations, since two matrix vector products (one with $A$ and one with $A^T$) are computed at each GKB iteration; the cost of computing the quantity $(\beta B_k B_k^T+I)^{-1}e_1$ is $O(k)$ floating point operations \sgnew{(exploiting the tridiagonal structure of the involved matrices)}, which is negligible. 
However, it is still unclear how an approximate solution $x_k(\alpha_{k+1})$ for problem (\ref{tikh}) can be computed. To achieve this, one needs to consider the space $\Kry_k(A^TA,A^Tb)$, and project problem (\ref{tikh}) onto it, i.e., solve the problem (\ref{hybridprojpb}) (with $\alpha_{k}=1/\beta_{k+1}$). This can be done inexpensively once the bound $\mcG_k$ is computed, since $k$ iterations of the GKB algorithm generate both spaces $\Kry_k(A^TA,A^Tb)$ and $\Kry_k(AA^T,b)$ (see Section \ref{sec:Hybrid}), and only the order-$k$ least squares problem in (\ref{hybridprojpb}) 
needs to be solved to \sgnew{compute $\widehat{y}_k(\beta_{k+1}):=y_k(1/\alpha_{k+1})\in\R^k$ ($O(k)$ floating point operations) and to then form $\widehat{x}_{k}(\beta_{k+1})=x_k(1/\alpha_{k+1})$ 
($O(kn)$ floating point operations). The cost of these computations is negligible if} 
$k\ll\min\{n,m\}$. 
According to Algorithm \ref{alg:new}, the task of solving problem (\ref{hybridprojpb}) can be performed only once a stopping criterion for the iterations is satisfied. However, \mslnew{$\widehat{y}_k(\beta_{k+1})\in\R^k$} may be needed to devise a suitable stopping criterion, and therefore should be computed at a negligible additional cost at each iteration. 
Indeed, it should be remarked that the discrepancy functional $\hfdp$ (\ref{eq:discrcvx}) associated to the approximate solution $\widehat{x}_k(\beta_{k+1})$ satisfies 
\begin{equation}\label{eq:discrmatrboundup}
\|b\|^2e_1^T(\beta_{k+1} \bB_k \bB_k^T+I)^{-2}e_1\ - \eps^2=:\mathcal{R}_{k+1}(\hphidp, AA^T,b, \beta_{k+1}), \quad k\geq 1\,;
\end{equation}
see (\ref{GQ2}) (full derivations are provided in \cite{Calvetti1999}). Since $\partial^{(2k+1)}_t \hphidp (t,\beta) > 0$ for all $k\geq 1$, $t\geq 0$, $\beta>0$, $\mathcal{R}_{k+1}$  is an upper bound for $\hfdp$ (both considered as functions of $\beta$); see (\ref{GRError}). 
Summarizing, the new adaptive strategy to solve problem (\ref{GenP}) when $P(x(\alpha))$ is the discrepancy principle consists in applying the new modified Newton zero finder (\ref{newtonstepdiscr}) to (\ref{eq:discrmatrbeta}), which involves computing lower bounds $\mcG_k$ for $\hfdp$ at the $k$th iteration of Algorithm \ref{alg:new}; the discrepancy functional \linebreak[4]\mslnew{$\|b-A\widehat{x}_k(\beta_{k+1})\|^2-\eps^2=\| \bB_k\widehat{y}_k(\beta_{k+1})-\|b\|e_1\|^2-\eps^2$  for the} intermediate approximate solutions $\widehat{x}_k(\beta_{k+1})=x_k(1/\alpha_{k+1})$ of problem (\ref{tikh}) lays on upper bounds $\mcR_{k+1}$ for $\hfdp$. 
An illustration of the behavior of the 
bounds for the function $\hfdp$ (\ref{eq:discrcvx}) for the problem in (\ref{illustration}) is given in Figure \ref{fig:DPbounds}; this example is quite representative of the typical behavior found in other test problems. Please note that the functions $\hfdp(\beta)$, $\mcG_k(\beta)$, and $\mcR_{k+1}(\beta)$ do not look convex because of the logarithmic scale. 
\begin{figure}[htbp]
\centering
\includegraphics[width=5.5cm]{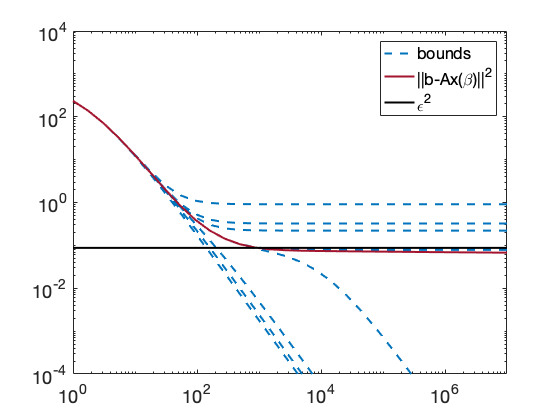} 
\caption{Values of the function $\|b-Ax(\beta)\|^2=\hfdp+\eps^2$ in (\ref{eq:discrcvx}), lower bounds $\mcG_k+\eps^2$ (\ref{eq:discrmatrbound}), and upper bounds $\mcR_{k+1}+\eps^2$ (\ref{eq:discrmatrboundup}) for $k=2,5,8,30$, versus $\beta$, for the problem in (\ref{illustration}); values  displayed in logarithmic scale.}
\label{fig:DPbounds}
\end{figure}

\paragraph{Stopping criteria for Algorithm \ref{alg:new}.}
Since the modified Newton method (\ref{newtonstepdiscr}) can essentially be regarded as a Newton-like update formula applied to a sequence of iteration-dependent converging functions, standard stopping criteria for Newton method can be adapted to this setting to determine both a value of the regularization parameter $\beta_{k+1}$ and the dimension of the approximation subspace for $\widehat{x}_k(\beta_{k+1})$. Typically, Algorithm \ref{alg:new} stops when the space $\Kry_k(A^TA,A^Tb)$ is large enough to contain a suitable approximation to the solution of (\ref{tikh}) and when a value of the regularization parameter suitable for the full dimensional problem (\ref{tikh}) has been computed: these requirements are interrelated and, as mentioned in Section \ref{sec:intro}, they are also desirable for hybrid methods; see \cite{Calvetti1999, Chung08, GNR15}. 

 

It is natural to stop Algorithm \ref{alg:new} as soon as 
\begin{equation}\label{eq:stop_th}
\hfdp(\beta_{k+1})= \frac{\|b-Ax(\beta_{k+1})\|^2-\eps^2}{\eps^2} \leq\tau\,,\quad\mbox{for a given tolerance $\tau>0$} \,.
\end{equation}  
However, computing $x(\beta_{k+1})$ would require solving the full-dimensional problem (\ref{tikh}) with $\alpha=1/\beta_{k+1}$, which may be prohibitively expensive for large-scale problems. Therefore, estimates for the numerator of the function on the left of (\ref{eq:stop_th}) should be considered. By taking the upper bound (\ref{eq:discrmatrboundup}), one can replace (\ref{eq:stop_th}) by
\begin{equation}\label{stopritDP}
 \mathcal{R}_{k+1}(\hphidp,AA^T,b,\beta_{k+1})\leq \tau \, \eps^2 \, .
\end{equation}
Satisfying (\ref{stopritDP}) implies satisfying (\ref{eq:stop_th}). Alternatively (and using reduced notations), an estimate of $\|b-Ax(\beta_{k+1})\|^2$ is obtained by averaging its upper (\ref{eq:discrmatrboundup}) and lower (\ref{eq:discrmatrbound}) bounds evaluated at $\beta_{k+1}$, and (\ref{eq:stop_th}) can be replaced by
\begin{equation}\label{stopritDPvar}
\frac{1}{2} \left( \mathcal{R}_{k+1}(\beta_{k+1})+\mathcal{G}_k(\beta_{k+1})\right)
\leq \tau \, \eps^2 \, .
\end{equation}

%

A different stooping criterion is devised by simultaneously monitoring the (relative) convergence of the sequence $\{\mathcal{G}_k \}_{k}$ to $\hfdp$ and the (relative) convergence of (\ref{newtonstepdiscr}) to the zero of $ \mathcal{G}_k(\beta)$, so that Algorithm \ref{alg:new} should be stopped as soon as 
\[
\frac{\|b-Ax(\beta_{k+1})\|^2 - \eps^2 - \mathcal{G}_k(\beta_{k+1})}{\|b-Ax(\beta_{k+1})\|^2-\eps^2}+\frac{\mathcal{G}_k(\beta_{k+1})}{ \eps^2} \leq \tau \,,
\quad\mbox{for a given tolerance $\tau>0$}\,.
\]
Analogously to (\ref{eq:stop_th}), to avoid excessive computations, the value of $\|b-Ax(\beta_{k+1})\|^2-\eps^2$ can be estimated by averaging its upper (\ref{eq:discrmatrboundup}) and lower (\ref{eq:discrmatrbound}) bounds evaluated at $\beta_{k+1}$, so that the following stopping rule is considered: 
\begin{equation}\label{stop_discr2}
\frac{\frac{1}{2}\left( \mathcal{R}_{k+1}(\beta_{k+1})+\mathcal{G}_k(\beta_{k+1})\right)- \mathcal{G}_k(\beta_{k+1})}{\frac{1}{2}\left( \mathcal{R}_{k+1}(\beta_{k+1})+\mathcal{G}_k(\beta_{k+1}) \right)}+\frac{\mathcal{G}_k(\beta_{k+1})}{ \eps^2} \leq \tol \, .
\end{equation}



\subsection{Other parameter rules}
This section explains how to approximate the solution of problem (\ref{GenP}) when considering the functionals (\ref{GCVtab}), (\ref{qotab}), and (\ref{Reginskatab}) defined in Table \ref{tab:rules}. The minimization procedure happens across the iterations of Algorithm \ref{alg:new} by applying a modified Newton method starting, in general, from an iteration $k\geq k^\ast$. Since the evaluation of $P(x(\alpha))$ at each iteration $k$ may be computationally prohibitive, one should employ a sequence of functionals $P_k(\alpha)$, which have a local minimum converging to a local minimum of $P(x(\alpha))$, but may not explicitly depend on the current approximate solution $x_k(\alpha)$. The functionals $P_k(\alpha)$ are obtained by projecting their full-dimensional counterparts onto Krylov subspaces of increasing dimension, or by approximating $P(x(\alpha))$ via Gaussian quadrature rules (see Section \ref{sec:Gauss}). For this reason, at the $k$th iteration of Algorithm \ref{alg:new}, the Krylov subspaces $\Kry_{k}(A^TA,A^Tb)$ and $\Kry_{k}(AA^T,b)$ are built, and 
the minimization step at line 4 of Algorithm \ref{alg:new} reads 
\begin{equation}\label{eq:modNewton_gen}
\alpha_{k+1}=\alpha_{k}-\frac{\partial_\alpha P_k(\alpha_k)}{\partial^{2}_{\alpha} P_k(\alpha_k)}\,,\quad\text{if}\quad k\geq k^\ast\,.
\end{equation}
Although the above relation is formally similar to (\ref{newtonstep2}), where $L_k=\partial_\alpha P_k$ and a zero finder is applied to $\partial_\alpha P(x(\alpha))=0$, 
applying (\ref{eq:modNewton_gen}) to the functionals (\ref{GCVtab}), (\ref{qotab}), and (\ref{Reginskatab}) 
is not as straightforward as in Section \ref{sec:discr}, for a variety of reasons: firstly, $\partial_\alpha P(x(\alpha))$ may have multiple zeros (corresponding to local maxima or minima of $P(x(\alpha))$; secondly, $\{\partial_\alpha P_k(\alpha)\}_k$ may not be nested upper or lower bounds for $\partial_\alpha P(x(\alpha))$; finally, some insight into how to choose $k^\ast$ is needed. Specific details are provided in the following subsections. 
%
Similarly to Section \ref{sec:discr}, an approximate solution $x_k(\alpha_{k+1})=V_ky_k(\alpha_{k+1})\in\Kry_k(A^TA,A^Tb)$ for problem (\ref{tikh}) can be computed by solving problem (\ref{hybridprojpb}) (with $\alpha_{k}=\alpha_{k+1}$, i.e., taking the most recent regularization parameter approximation from (\ref{eq:modNewton_gen})): even if, according to Algorithm \ref{alg:new}, this can be done only after a stopping criterion for $k$ is satisfied, $y_k(\alpha_{k+1})\in\R^k$ may be needed to devise an appropriate stopping criterion. 
If $k\ll\min\{n,m\}$, the computational cost of performing $k$ iterations of Algorithm \ref{alg:new} is dominated by the cost of performing $k$ GKB iterations, i.e., $O(2kmn)$ floating point operations. 
%

\subsubsection{Generalized cross validation (GCV)}\label{sec:GCV}
The functional associated to generalized cross validation (GCV) in Table \ref{tab:rules} can be expressed as 
\begin{equation}\label{GCValpha}
P(x(\alpha))=  \frac{b^T \phigcvone(AA^T,\alpha) b}{\text{trace}\left( \phigcvtwo(AA^T,\alpha) \right)^2} \,,\quad\mbox{where}\quad
\begin{array}{ccl}
\phigcvone(t,\alpha)&=&\alpha^2 (\alpha + t)^{-2} \\
\phigcvtwo(t,\alpha)&=&\alpha (\alpha + t)^{-1}
\end{array}\,,
\end{equation}
i.e., $P(x(\alpha))$ is the ratio of a quadratic form and the trace of a function of a matrix. Note that, since $P(x(\alpha))$ is typically quite flat around its minimum, minimizing (\ref{GCValpha}) with respect to $\alpha$ is challenging; 
see, e.g., \cite{Fenu16,NoRu14}. Because of this, when applying the modified Newton method (\ref{eq:modNewton_gen}), one should be careful in devising an appropriate sequence of $P_k(\alpha)$. Different approaches can be found in the literature for approximating $P(x(\alpha))$: while lower and upper bounds can be easily derived for its numerator using Gaussian quadrature rules (note that $\phigcvone(t,\alpha)=\phidp(t,\alpha)+\bar{\eps}^2$; see (\ref{eqn:psi})), finding an approximation for the trace in the denominator is a well-studied but difficult task. One could, for instance, use random estimators for the trace; see, e.g., \cite{VonMatt,randtrace}. A method for computing bounds for $P(x(\alpha))$ based on multiple runs of the so-called global Golub-Kahan algorithm is presented in \cite{Fenu16}. 
When performing hybrid methods, i.e., when solving a sequence of problems (\ref{GenPproj}), it is common to take as denominator of the functional $P_k(x_k(\alpha))$ the square of the quantity 
\begin{equation}\label{GCVhybrid}
(m-k) + \text{trace}\left( \phigcvtwo(\bB_k\bB_k^T,\alpha) \right)\,, 
\end{equation}
where $\bB_k$ is defined in (\ref{eq:Lanczos1}) 
(basically, $\left(\text{trace}(I - U_{k+1}\bB_k(\bB_k^T\bB_k+\alpha I)^{-1}\bB_k^TU_{k+1}^T)\right)^2$ is considered at the denominator of (\ref{GCVtab}) in Table \ref{tab:rules}; see \cite{Chung08,NoRu14,RenautHybrid}). 

In this paper, the following functional
\begin{equation} \label{eqn:GCVtrace}
P_k(\alpha)= \frac{\|b \|\sg{^2} \phigcvone(\bC_k \bC_k^T,\alpha)}{\text{trace}\left( \phigcvtwo(\bC_k \bC_k^T,\alpha) \right)^2}=\frac{\mathcal{R}_{k+1}(\phigcvone, A A^T, b, \alpha)}{\text{trace}\left( \phigcvtwo(\bC_k \bC_k^T,\alpha) \right)^2}
\end{equation}
is considered at the $k$th iteration of Algorithm \ref{alg:new} as an approximation to (\ref{GCValpha}). Basically, decreasing upper bounds are considered for the numerator of (\ref{GCValpha}) using Gauss-Radau quadrature rules (see Section \ref{sec:discr}). The denominator is heavily 
under-estimated by squaring the trace of the matrix function $\phigcvtwo$ evaluated at $\bC_k \bC_k^T$ (i.e., by taking the same approximation (\ref{GCVhybrid}) used for hybrid methods, without the $(m-k)$ term); under some assumptions on $\alpha$ and the entries of $\bC_k$ (which are typically satisfied for ill-posed problems), the denominator increases with $k$. 
%
%
%
%
%
%
%
%
Because of the loose trace estimator, the sequence $\{P_k(\alpha)\}_{k}$ defined in (\ref{eqn:GCVtrace}) is not required to converge to $P(x(\alpha))$ within the performed number of iterations. An illustration of the behavior of 
the GCV functional (\ref{GCValpha}) together with its upper bounds (\ref{eqn:GCVtrace}) for the problem in (\ref{illustration}) is given in Figure \ref{fig:approxGCV}. Note that, even after 225 iterations, the bound $P_{225}(\alpha)$ is very different from $P(x(\alpha))$; this example is quite representative of the typical behavior found in other test problems. 
\begin{figure}[htbp]
\centering
\includegraphics[width=5.5cm]{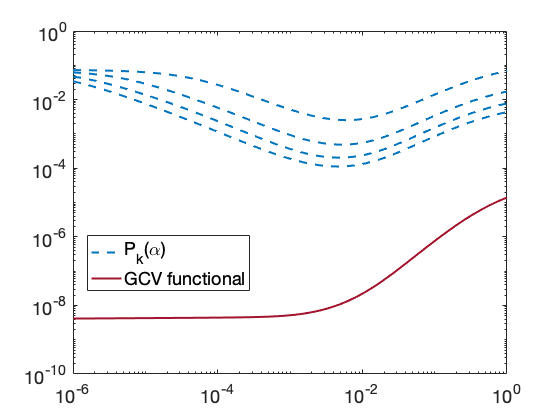} 
\caption{Values of the GCV functional (\ref{GCValpha}) and upper bounds (\ref{eqn:GCVtrace}) for $k=1,75,150, 225$, versus $\alpha$, for the problem in (\ref{illustration}); values displayed in logarithmic scale.}
\label{fig:approxGCV}
\end{figure}
When $P_k(\alpha)$ is employed as higher-level objective function in (\ref{GenPproj}), i.e., for hybrid methods, 
considering estimate (\ref{eqn:GCVtrace}) can lead to oversmoothed approximate solutions $x_k(\alpha)$; see \cite{Chung08,RenautHybrid}. However, in the framework of Algorithm \ref{alg:new}, considering (\ref{eqn:GCVtrace}) is convenient because each $P_k(\alpha)$ is less flat around its (local) minimum (and therefore less challenging to minimize). Indeed, by monitoring the stabilization of the parameter $\alpha_{k+1}$ selected by (\ref{eq:modNewton_gen}) across the iterations of Algorithm \ref{alg:new}, one can make sure that the location of the approximate (local) minimum of $P_k(\alpha)$ stabilizes for the subsequent functionals (\ref{eqn:GCVtrace}), which is a necessary (but not sufficient) condition for $\alpha_k$ to belong to a neighborhood of $\text{arg}\min_{\alpha} P(x(\alpha))$; see Section~\ref{subsec:stop} and Section~\ref{sec:numerics}. As suggested in \cite{VonMatt}, the choice $k^\ast=\lceil 3\log(\min\{m,n\}) \rceil$, where $\lceil\cdot\rceil$ denotes the ceiling function, is made in (\ref{eq:modNewton_gen}) to allow the \mslnew{approximations} (\ref{eqn:GCVtrace}) to slightly stabilize (especially for $\alpha$ small) before applying the modified Newton method.

\subsubsection{Quasi-optimality criterion}
The functional associated to the quasi-optimality criterion (\ref{qotab}) in Table \ref{tab:rules} can be expressed in terms of a quadratic form as
\begin{equation}\label{eq:minQO}
P(x(\alpha))= (A^T b)^T \phiqo(A^T A,\alpha) A^T b \,,\quad\mbox{where}\quad\phiqo(t,\alpha)=\alpha^2(\alpha + t)^{-4}.
\end{equation}
Since $\partial^{(2k-1)}_t \phiqo(t,\alpha) < 0 $ for $k\geq 1$, $t \geq 0$, $\alpha > 0$, Gauss-Radau quadrature rules can be used to compute upper bounds 
for $P(x(\alpha))$ in (\ref{eq:minQO}). Namely, the quadratic forms
\begin{equation}\label{qoup}
\mcR_k(\alpha):=\mathcal{R}_k ( \phiqo, A^TA, A^T b, \alpha) = \|A^T b \|^2 e^T_1 \phiqo(\hbC_{k-1} \hbC_{k-1}^T,\alpha) e_1
\end{equation} 
are such that 
$\mcR_{k+1}(\alpha)\leq \mcR_k(\alpha)$ for all $k\geq 2$ and $\alpha>0$ (see \cite[Theorem 2.2]{Lopez08}), and $\mcR_k(\alpha)=P(x(\alpha))$ for $k\geq \min\{n,m\}$. 
Similarly, since $\partial^{(2k)}_t \phiqo(t,\alpha) > 0 $ for all $k\geq 1$, $t \geq 0$, $\alpha > 0$, Gauss quadrature rules can be used to compute lower bounds for $P(x(\alpha))$ in (\ref{eq:minQO}). Namely, the quadratic forms
\begin{equation}\label{qolow}
\mcG_k(\alpha):=\mathcal{G}_k ( \phiqo, A^TA, A^T b, \alpha) = \|A^T b \|^2 e^T_1 \phiqo(\hC_k \hC_k^T,\alpha) e_1
\end{equation} 
are such that 
$\mcG_{k+1}(\alpha)\geq \mcG_k(\alpha)$ for all $\alpha>0$ (see (\ref{GQ2}) and \cite[Theorem 2.2]{Lopez08}), and $\mcG_k(\alpha)=P(x(\alpha))$ for $k\geq \min\{n,m\}$. An illustration of the behavior of 
the quasi-optimality functional (\ref{eq:minQO}) together with its upper bounds (\ref{qoup}) and lower bounds (\ref{qolow}) for the problem in (\ref{illustration}) is given in Figure \ref{fig:QObounds}; this example is quite representative of the typical behavior found in other test problems. 
 \begin{figure}[htbp]
\centering
\includegraphics[width=5.5cm]{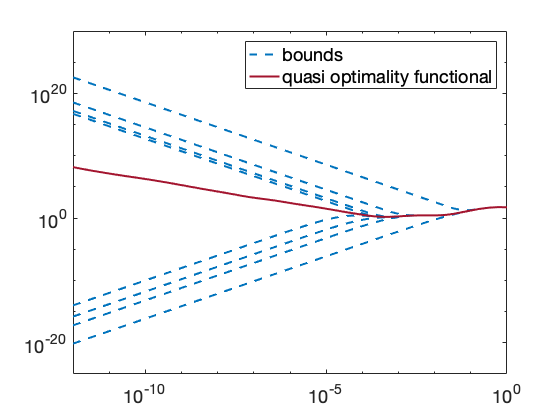} 
\caption{Values of the quasi optimality functional (\ref{eq:minQO}), and some of its lower bounds (\ref{qolow}) and upper bounds (\ref{qoup}) for $k=5,35,65,95$, versus $\alpha$, for the problem in (\ref{illustration}); values displayed in logarithmic scale.}
\label{fig:QObounds}
\end{figure}
Looking at Figure \ref{fig:QObounds}, it is evident that the functional in (\ref{eq:minQO}) is nonconvex and quite flat around its local minima, so that computing its global minimizer can be difficult. However, its upper bounds display a much more favorable behavior when it comes to optimization: for small $k$ and for small $\alpha$ the functionals in (\ref{qoup}) are monotonically and quickly decreasing. For this reason, at the $k$th iteration of Algorithm \ref{alg:new}, the choice $P_k(\alpha)=\mcR_k(\alpha)$ is made. The modified Newton method (\ref{eq:modNewton_gen}) rapidly leads to the computation of local minima for each $P_k(\alpha)$, and therefore of a local minimum for $P(x(\alpha))$ (see also Section \ref{sec:numerics}). Note that, since $P_k(\alpha)$ is defined for $k\geq 2$ (see (\ref{GRQ2})), it is natural to select $k^\ast=2$ in (\ref{eq:modNewton_gen}). The lower bounds (\ref{qolow}) for $P(x(\alpha))$ are typically very flat (recall that the graphs in Figure \ref{fig:QObounds} are displayed in logarithmic scale), and they can be used to devise suitable stopping criteria for Algorithm \ref{alg:new} (see Section \ref{subsec:stop}, where the notation $P^L_k(\alpha)=\mathcal{G}_k (\alpha)$ is used). 
%
%

\subsubsection{Regi\'{n}ska criterion}
The functional associated to the Regi\'{n}ska criterion (\ref{Reginskatab}) in Table \ref{tab:rules} can be expressed in terms of quadratic forms as
\begin{equation}\label{eq:minR}
P(x(\alpha)) = \sqrt{b^T\phiR(AA^T,\alpha)b} \sqrt{(A^T b)^T\phiR(A^T A,\alpha)A^T b}, \quad \mbox{where}\quad \phiR(t,\alpha)=\alpha(\alpha + t)^{-2}. 
\end{equation}
Since $\partial^{(2k-1)}_t \phiR(t,\alpha) < 0 $ for $k\geq 1$, $t \geq 0$, $\alpha > 0$, Gauss-Radau quadrature rules can be used to compute a sequence of increasingly sharper upper bounds for $b^T\phiR(AA^T,\alpha)b$ and $(A^T b)^T\phiR(A^T A,\alpha)A^T b$. Namely, taking
\begin{eqnarray}
\widetilde{\mcR}_k(\alpha) &:=&  \sqrt{\mathcal{R}_{k+1}(\phiR, A A^T, b,\alpha)} \sqrt{\mathcal{R}_{k}(\phiR, A^T A, A^T b, \alpha)} \nonumber
 \\ 
&=& \|A^T b \| \| b \| \sqrt{e^T_1 \phiR(\bC_k \bC_k^T,\alpha) e_1} \sqrt{e^T_1 \phiR(\hbC_{k-1} \hbC_{k-1}^T,\alpha) e_1}  \label{rup}
\end{eqnarray}
and knowing that both $\mathcal{R}_{k+1}(\phiR, A A^T, b,\alpha)$ and $\mathcal{R}_{k+1}(\phiR, A^T A, A^T b, \alpha)$ decrease with  increasing $k\geq 2$, 
one gets $\widetilde{\mcR}_{k+1}(\alpha)\leq \widetilde{\mcR}_{k}(\alpha)$ (see Section \ref{sec:discr} and equations (\ref{GRQ1}), (\ref{GRQ2})). Similarly, since \linebreak[4]
$\partial^{(2k)}_t \phiR(t,\alpha)>0 $ for all $k\geq 1$, $t \geq 0$, $\alpha > 0$, Gauss quadrature rules can be used to compute lower bounds for $P(x(\alpha))$ in (\ref{eq:minR}). Namely, taking
\begin{eqnarray}
\widetilde{\mcG}_k(\alpha) &:=&  \sqrt{\mathcal{G}_k(\phiR, A A^T, b,\alpha)} \sqrt{\mathcal{G}_k(\phiR, A^T A, A^T b, \alpha)} \nonumber \\ 
&=& \|A^T b \| \| b \| \sqrt{e^T_1 \phiR(B_k B_k^T,\alpha) e_1}  \sqrt{e^T_1 \phiR(\hC_k \hC_k^T,\alpha) e_1} \label{rlow}
\end{eqnarray}
and knowing that 
\[
\mathcal{G}_{k+1}(\phiR, A A^T, b,\alpha)\geq\mathcal{G}_{k}(\phiR, A A^T, b,\alpha)\mbox{ and }\mathcal{G}_{k+1}(\phiR, A^T A, A^T b, \alpha)\geq\mathcal{G}_{k}(\phiR, A^T A, A^T b, \alpha)\,,
\]
one gets $\widetilde{\mcG}_{k+1}(\alpha)\geq \widetilde{\mcG}_{k}(\alpha)$ (see equations (\ref{GQ1}) and (\ref{GQ2})). 
An illustration of the behavior of the Regi\'{n}ska's functional (\ref{eq:minR}) together with its upper bounds (\ref{rup}) and lower bounds (\ref{rlow}) for the problem in (\ref{illustration}) is given in Figure \ref{fig:Rbounds}; this example is quite representative of the typical behavior found in other test problems. 
\begin{figure}[htbp]
\centering
\includegraphics[width=5.5cm]{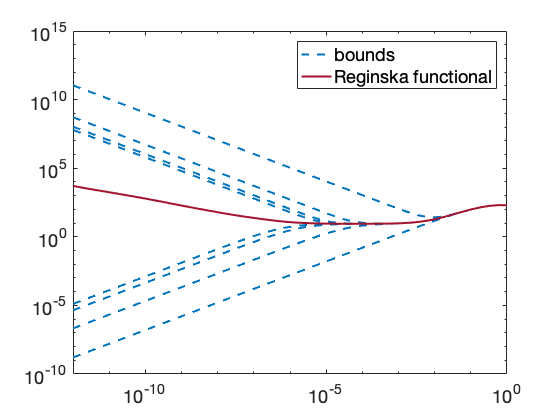} 
\caption{Values of the Regi\'{n}ska's functional (\ref{eq:minR}), and some of its lower bounds (\ref{rlow}) and upper bounds (\ref{rup}) for $k=5,35,65,95$, versus $\alpha$, for the problem in (\ref{illustration}); values displayed in logarithmic scale.}
\label{fig:Rbounds}
\end{figure}
Analogously to the quasi-optimality case, looking at Figure \ref{fig:Rbounds} it is evident that the functional in (\ref{eq:minR}) is quite difficult to minimize numerically because it is quite flat around its miniumum. Since its upper bounds (\ref{rup}) are less flat around its minima for small values of $k$, at the $k$th iteration of Algorithm \ref{alg:new} the choice $P_k(\alpha)=\widetilde{\mcR}_k(\alpha)$ is made. The modified Newton method (\ref{eq:modNewton_gen}) rapidly leads to the computation of a minimum for $P(x(\alpha))$ (see also Section \ref{sec:numerics}). Note that, since $P_k(\alpha)$ is defined for $k\geq 2$ (see (\ref{GRQ2})), it is natural to select $k^\ast=2$ in (\ref{eq:modNewton_gen}). The lower bounds (\ref{rlow}) for $P(x(\alpha))$ are typically very flat and they can be used to devise suitable stopping criteria for Algorithm \ref{alg:new} (see Section \ref{subsec:stop}, where the notation $P^L_k(\alpha)=\widetilde{\mathcal{G}}_k (\alpha)$ is used). 

\subsubsection{Stopping criteria for Algorithm \ref{alg:new}}\label{subsec:stop}
Similarly to Section \ref{sec:discr}, traditional stopping criteria for (standard) Newton method applied to $\partial_{\alpha}P(x(\alpha))=0$ can be adapted to the modified Newton method (\ref{eq:modNewton_gen}). Namely, one could stop the iterations as soon as the following condition on the relative residual is satisfied
\[ 
\frac{|\partial_{\alpha} P(x(\alpha_{k+1}))|}{|P(x(\alpha_{k+1}))|} < \tol\,,\quad\mbox{for a given tolerance $\tau>0$} \,.
\]
However, computing $x(\alpha_{k+1})$ would require solving the full-dimensional problem (\ref{tikh}) for \linebreak[4]$\alpha=\alpha_{k+1}$, which could be computationally infeasible for large-scale problems. 

Moreover, when applying the modified Newton method (\ref{eq:modNewton_gen}) to compute an approximate solution of (\ref{GenP}), 
one should at least jointly  monitor the stabilization 
of the parameter $\alpha_{k+1}$ and the convergence of $\alpha_{k+1}$ to a zero of $P_k(\alpha)$. Namely the iterations are stopped as soon as
\begin{equation} \label{eqn:sc1}
\frac{|\alpha_{k+1}-\alpha_{k}|}{\frac{1}{2}|\alpha_{k+1}+\alpha_{k}|}+\frac{|\partial_{\alpha} P_k(\alpha_{k+1})|}{|P_k(\alpha_{k+1})|} < \tol\,,\quad\mbox{for a given tolerance $\tau>0$} \,. 
\end{equation}
The first term on the left-hand side of (\ref{eqn:sc1}) is the relative change in two consecutively computed values of $\alpha$ (the quantity at the denominator is their average), while the second term measures the relative residual of the approximation $\partial_\alpha P_k(\alpha)$ evaluated in $\alpha_{k+1}$ (i.e., the distance of $\partial_\alpha P_k(\alpha_{k+1})$ to $0$). 

Finally, when $\{P_k(\alpha)\}_k$ are approximations of $P(x(\alpha))$ of improving quality (i.e., when $P_k(\alpha)$ becomes closer to $P(x(\alpha))$ as $k$ increases), one can also monitor the convergence of $P_k(\alpha_{k+1})$ to $P(x(\alpha_{k+1}))$ (i.e., the convergence of the approximate functionals to the full-dimensional one at the current approximation of $\alpha$) together with the convergence of $\alpha_{k+1}$ to a zero of $P_k(\alpha)$. This should happen when dealing with the quasi-optimality and Regi\'{n}ska criteria, although it may not happen when dealing with GCV (see Section \ref{sec:GCV}, where the case for considering as $\{P_k(\alpha)\}_k$ coarse over-estimations of $P(x(\alpha))$ is made). In these cases, the iterations are stopped as soon as
\begin{equation} \label{eqn:sc2}
\frac{|P_{k}(\alpha_{k+1})-\bar{P}_k(\alpha_{k+1})|}{|\bar{P}_k(\alpha_{k+1})|}+\frac{|\partial_{\alpha} P_k(\alpha_{k+1})|}{|P_k(\alpha_{k+1})|} < \tol\,,\quad\mbox{for a given tolerance $\tau>0$} \,,
\end{equation}
where
\begin{equation}\label{barP}
\bar{P}_k(\alpha_{k+1})=\frac{1}{2}\left(P_{k}(\alpha_{k+1})+P_k^{L}(\alpha_{k+1})\right)\approx P(x(\alpha_{k+1}))\,.
\end{equation}
Recall that, for the quasi-optimality and Regi\'{n}ska criteria: $P_k(\alpha) = \mcR_k(\alpha)$ (see (\ref{qoup})) and $P_k(\alpha) = \widetilde{\mcR}_k(\alpha)$ (see (\ref{rup})), respectively (both of them are upper bounds for \linebreak[4]$P(x(\alpha))$); $P^L_k(\alpha) = \mcG_k(\alpha)$ (see (\ref{qolow})) and $P^L_k(\alpha) = \widetilde{\mcG}_k(\alpha)$ (see (\ref{rlow})), respectively (both of them are lower bounds for $P(x(\alpha)))$. The quantity (\ref{barP}), i.e., the average of two approximations of $P(x(\alpha_{k+1}))$ is used to avoid computing $P(x(\alpha_{k+1}))$ itself, which may be demanding.

\section{Numerical Experiments}\label{sec:numerics}
This section investigates the performance of the proposed adaptive regularization parameter choice rules on two large-scale test problems modeling imaging applications. All the experiments are performed running MATLAB R2017a and using some of the functionalities available within the MATLAB toolbox \emph{IR Tools}  \cite{IRtools}. The behavior of the approximate solution of problem (\ref{GenP}) for a range of regularization parameters $\alpha$ and dimensions $k$ of the Krylov subspaces can be monitored by checking the values of the relative restoration error 
\begin{equation}\label{relerr}
\RRE(\alpha,k)= \frac{\| x_k(\alpha)-x^{\text{ex}}\|}{\|x^{\text{ex}}\|}\,,
\end{equation}
where $x^{\text{ex}}$ is the exact solution of the noise-free version of problem (\ref{eq:linsys}) (i.e., $e=0$). This can be conveniently visualized by means of three-dimensional representations, where some sampled values of $\alpha$ and $k$ are reported on the $x$ and $y$ axes, respectively, and the corresponding values of $\RRE(\alpha,k)$ are reported on the $z$ axis. These plots are dubbed \emph{error surfaces}: points laying on the \emph{error surfaces} have coordinates $(\bar{\alpha},\bar{k},\RRE(\bar{\alpha},\bar{k}))$, where $\bar{\alpha}$ and $\bar{k}$ are sampled values of $\alpha$ and $k$, respectively. Similarly, the behavior of the functionals used in line 4 of Algorithm \ref{alg:new} to approximate the higher-level objective function $P(x(\alpha))$ in (\ref{GenP}) can be plotted against sampled values of $\alpha$ and $k$, giving rise to so-called \emph{higher-level surfaces}. Consistently with traditional representations, the points laying on the \emph{higher-level surfaces} associated to the GCV, the quasi-optimality and Regi\'{n}ska criteria have coordinates $(\bar{\alpha},\bar{k},P_{\bar{k}}(\bar{\alpha}))$. The points laying on the  \emph{higher-level surfaces} associated to the discrepancy principle have coordinates $(\bar{\alpha},\bar{k},\|b-Ax_{\bar{k}}(\bar{\alpha})\|^2)$: note that the values on the $z$-axis do not approximate the functional $P(x(\alpha))$ in Table \ref{tab:rules}, and they correspond to $\mcR_{\bar{k}+1}(\phidp,AA^T,b,\bar{\alpha}) + \eps^2$ (while, at line 4 of Algorithm \ref{alg:new}, the functional $\mathcal{G}_k(\hphidp, AA^T,b, 1/\alpha)$ is used; see Corollary \ref{cor1}). It is important to remark that the computation of \emph{higher-level surfaces} merely has illustrative purposes: the $k$th iteration of Algorithm \ref{alg:new} requires evaluating the approximated higher-level functionals in a given $\alpha$ only twice (see recursions (\ref{newtonstepdiscr}) and (\ref{eq:modNewton_gen})), until a stopping criterion is satisfied (the stopping criterion may require one or two extra evaluations of the approximated functionals). 

Notation-wise, in this section the value of the regularization parameter $\alpha$ computed at the $k$th step of an iterative procedure is denoted by $\alpha_k$: note that, in the framework of Algorithm \ref{alg:new}, this was denoted by $\alpha_{k+1}$ (see (\ref{newtonstepdiscr}), where $\alpha_{k+1}=1/\beta_{k+1}$, and (\ref{eq:modNewton_gen})), while in the framework of Algorithm \ref{alg:A2} this was denoted by $\alpha_k=\alpha_{k(\ell)}$. 

Algorithm \ref{alg:new} needs to be initialized by selecting the first value of the regularization parameter $\alpha_1$. In the discrepancy principle case, it is important that the initial $\beta_1=1/\alpha_1$ is such that $\mcG_1(\beta_1)\geq 0$, to satisfy the assumptions of Corollary \ref{cor1} and guarantee convergence: for this reason the value $\beta_1=10^{-10}$ is considered in the following experiments; also, $\eps=1.01\cdot\|e\|$ is set. For the other parameter rules, the initial value of $\alpha_1$ is less critical, since no clear convergence theory has been established: the value $\alpha_1=10^{-10}$ is considered in the following experiments. The tolerance employed for all the stopping criteria is $\tol=10^{-2}$. 
%
%

\paragraph{Example 1.}  An image deblurring test problem involving a \emph{satellite} test image of size 256$\times$256 pixels, a medium Gaussian blur, and Gaussian white noise level $\|e\|/\|b\|=10^{-2}$ is generated using the following instructions within \emph{IR Tools}:
\[
\text{\texttt{[A,b,x] = PRblur(256); bn = PRnoise(b);}}
\]
The coefficient matrix $A$ has order $65536$. Figure \ref{fig:DEBLURRING_graphs} evaluates the performance of Algorithm \ref{alg:new} by comparing the value of the regularization parameter $\alpha_k$ and the relative restoration error $\RRE(\alpha_k,k)$ 
computed at each iteration $k$, to the optimal ones and to the ones obtained running traditional hybrid methods (Algorithm \ref{alg:A2}); all the parameter choice rules listed in Table \ref{tab:rules} are tested. The optimal $\alpha_k$ is the one that minimizes $\RRE(\alpha,k)$ at each (fixed) iteration $k$, among a set of sampled $\alpha$. For all the considered parameter choice rules, the values of the parameter $\alpha_k$ computed by Algorithm \ref{alg:new} clearly and efficiently converge to the values obtained by applying hybrid methods. 
Even if some of the considered strategies seem to deliver relative reconstruction errors that are closer to the optimal ones, comparing the performance of each parameter choice rule with respect to the optimal one is beyond the scope of this paper.  
For most of the considered parameter choice rules, the stopping criteria proposed in Section \ref{sec:new} succeed in stopping the iterations of Algorithm \ref{alg:new} once a good regularization parameter is computed (i.e., when the couple $(\alpha,k)$ coincides with one that can be selected by a traditional hybrid method, while being quite close to the optimal one). Even when a stopping criterion seems to fail (meaning that no stopping happens within the maximum number of performed iterations or, looking a posteriori at the plots of $\alpha_k$ and $\RRE(\alpha_k,k)$ versus $k$, the iterations could have been stopped earlier), the quality of the approximate solution does not deteriorate with respect to the one achieved at a more computationally convenient stopping point (see, e.g, frames (b) and (d) of Figure \ref{fig:DEBLURRING_graphs}).
%
Also, it is evident that both hybrid methods and Algorithm \ref{alg:new} over-regularize the solution during their early iterations (i.e., they select a Tikhonov parameter that is much larger than the optimal one). 
%
%
%
\begin{table}
\centering
\caption{Summary of the markers denoting different stopping criteria for Algorithm \ref{alg:new}.}
\label{tab:markers} 
\footnotesize
\begin{tabular}{|cc|cc|cc|cc|cc|cc|}\hline
\multicolumn{6}{|c|}{\textbf{discrepancy principle}} & \multicolumn{2}{c|}{\textbf{GCV}} & \multicolumn{4}{c|}{\textbf{other rules}}\\\hline
(\ref{stopritDP}) & $\circ$ & (\ref{stopritDPvar}) & $\times$ & (\ref{stop_discr2}) & $\diamond$ & (\ref{eqn:sc1}) & $\circ$ & (\ref{eqn:sc1}) & $\circ$ & (\ref{eqn:sc2}) & 
$\diamond$\\ \hline
\end{tabular} 
\end{table}
%
%
\begin{figure}[htbp]
\centering
\begin{tabular}{cccc}
\vspace{0.2cm}
{\small {(1a)}} & {\small {(1b)}} & {\small {(1c)}} & {\small {(1d)}} \\ 
\includegraphics[width=3cm]{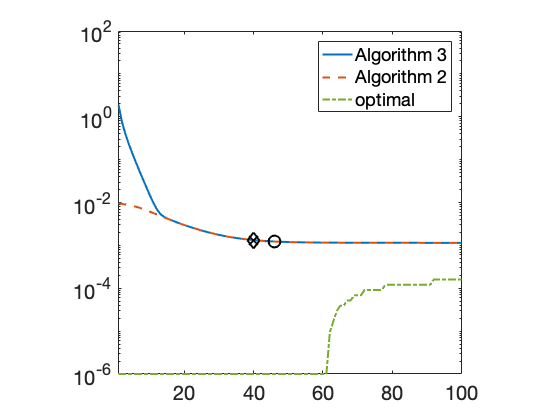} & 
\includegraphics[width=3cm]{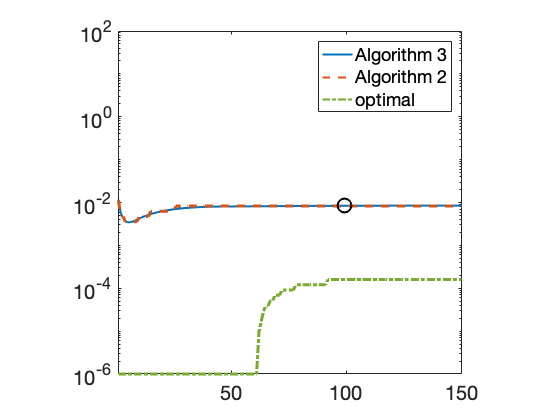} & 
\includegraphics[width=3cm]{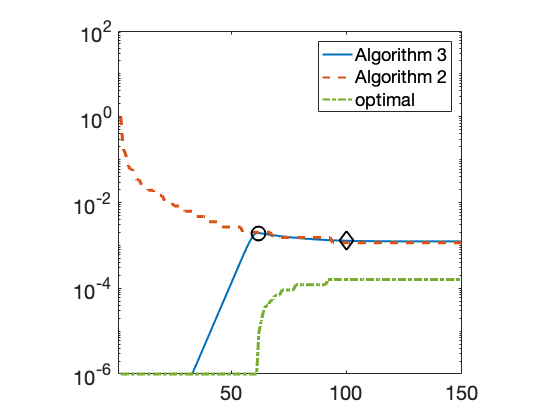}
&
\includegraphics[width=3cm]{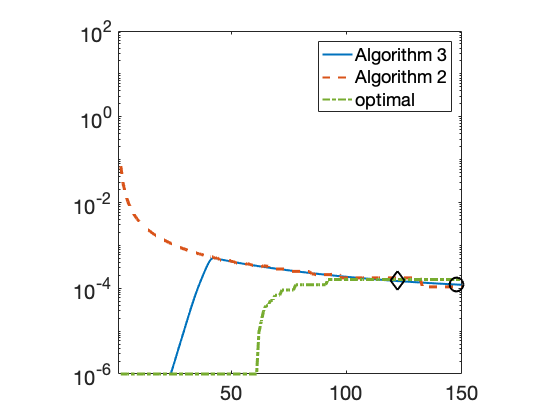} \\
{\small {(2a)}} & {\small {(2b)}} & {\small {(2c)}} & {\small {(2d)}} \\ 
\includegraphics[width=3cm]{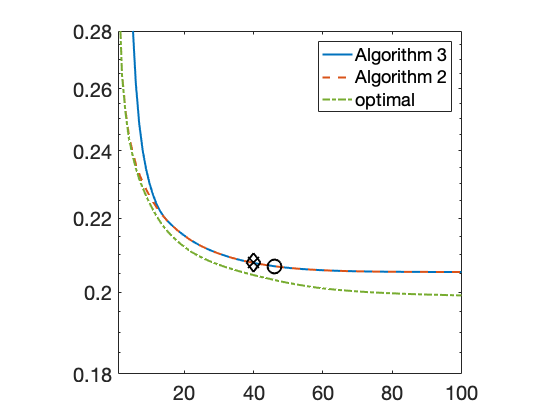} & 
\includegraphics[width=3cm]{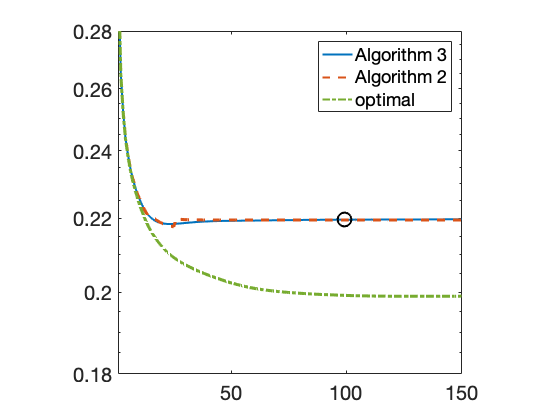} & 
\includegraphics[width=3cm]{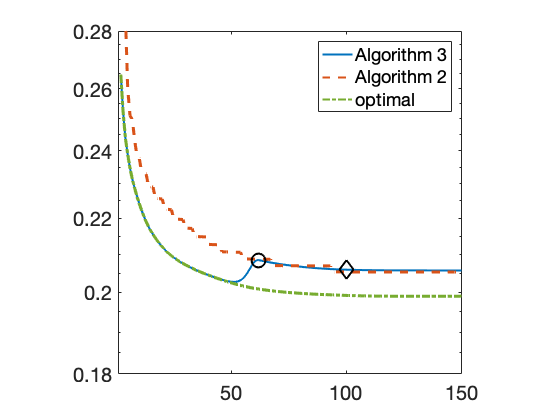}
&
\includegraphics[width=3cm]{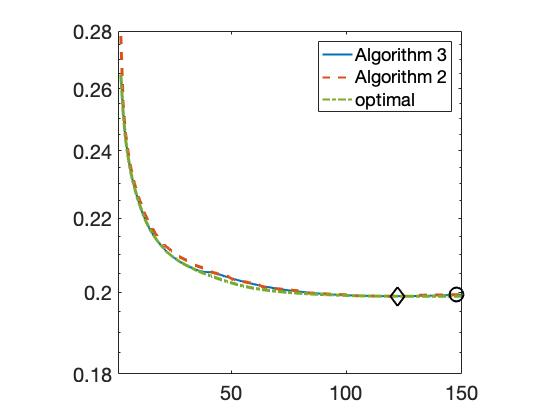}
\end{tabular}
\caption{Example 1. First row: values of the regularization parameter $\alpha_k$ versus number of iterations $k$. Second row: values of the relative restoration errors $\RRE(\alpha_k,k)$ versus number of iterations $k$. The following parameter choice rules are considered: (a) discrepancy principle; (b) GCV; (c) quasi-optimality criterion; (d) Regi\'{n}ska criterion. Special markers highlight the stopping iterations (with the conventions explained in Table \ref{tab:markers}).} \label{fig:DEBLURRING_graphs}
\end{figure}
Figure \ref{fig:DUBLURRING_Valleys} displays the \emph{higher-level surface} for this test problem and for the parameter choice rules listed in Table \ref{tab:rules}. 
Special markers are used to highlight the values corresponding to the couples $(\alpha_k,k)$ computed by Algorithms \ref{alg:A2} and \ref{alg:new}, and the optimal values. Further to Figure \ref{fig:DEBLURRING_graphs}, Figure \ref{fig:DUBLURRING_Valleys} displays how the computed approximations of the higher-level functionals $P(x(\alpha))$ listed in Table \ref{tab:rules} (or, in the discrepancy principle case, of its derivative with respect to $\alpha$) vary with respect to both $\alpha$ and $k$: for instance, it is evident that the shape of these surfaces stabilizes already for values of $k$ that are typically very small with respect to the dimension of the full-size problem. 
%
\begin{figure}[htbp]
\centering
\begin{tabular}{cc}
\vspace{0.2cm}
{\small {(a)}} & {\small {(b)}}\\ 
\includegraphics[width=5.5cm]{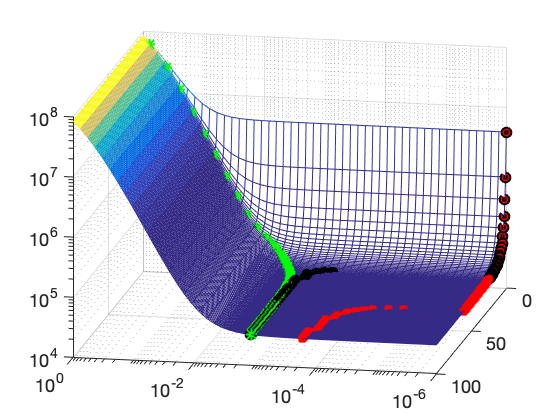} & 
\includegraphics[width=5.5cm]{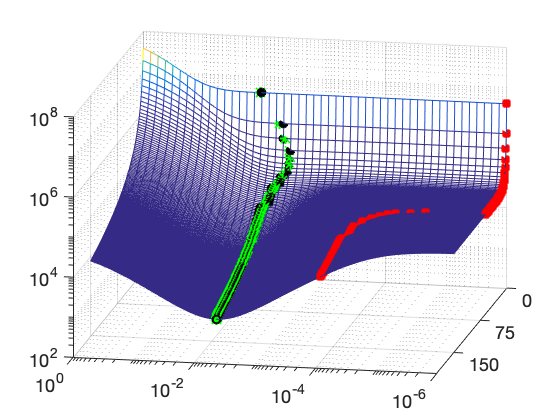} \\
{\small {(c)}} & {\small {(d)}}\\ 
\includegraphics[width=5.5cm]{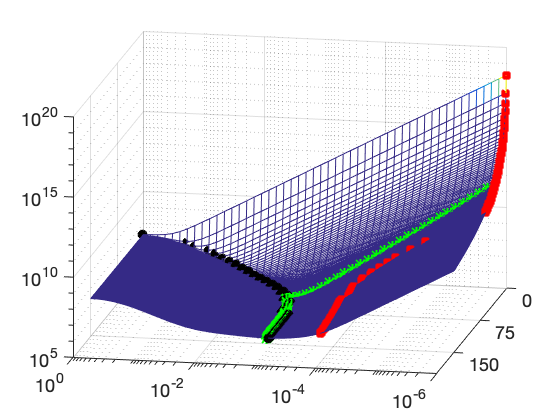} & 
\includegraphics[width=5.5cm]{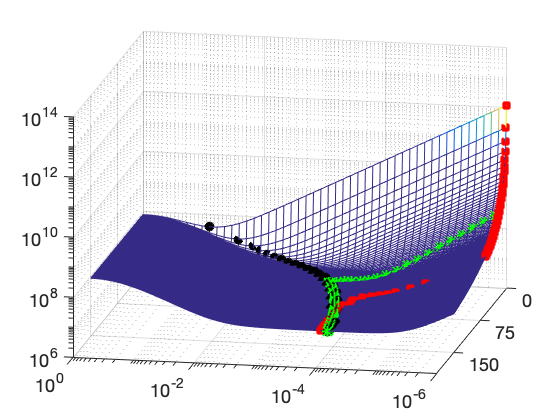}
\end{tabular}
\caption{\emph{Higher-level surface} for Example 1. 50 logarithmically equispaced values of $\alpha$ between $10^{-6}$ and $1$ and increasing values of $k$ up to 150 are considered. (a) Discrepancy principle (i.e., $\|b-Ax_k(\alpha)\|^2$); (b) GCV (i.e., upper bounds $P_k(\alpha)$ in (\ref{eqn:GCVtrace})); (c) quasi-optimality criterion (i.e., upper bounds $P_k(\alpha)$ in (\ref{qoup})); (d) Regi\'{n}ska criterion (i.e., upper bounds $P_k(\alpha)$ in (\ref{rup})). The green stars and black circles highlight the quantities computed by Algorithm \ref{alg:new} and Algorithm \ref{alg:A2}, respectively; the red squares denote the optimal quantities.} \label{fig:DUBLURRING_Valleys}
\end{figure}
Figure \ref{fig:DUBLURRING_OptimalValley} displays the \emph{error surface} for this test problem. 
The optimal value of the regularization parameters at each iteration is highlighted by a special marker. In agreement with Figure \ref{fig:DEBLURRING_graphs}, one can clearly see that the optimal regularization parameter is tiny when $k$ is small (as explained in \cite{Chung08}, this is due to the inherent regularizing effect of the GKB algorithm); also, once the Krylov subspace has reached a certain dimension (which, again, is typically very small compared to the dimension of the full-size problem), the optimal value of the regularization parameter stabilizes across subsequent iterations. Further to Figure \ref{fig:DEBLURRING_graphs}, Figure \ref{fig:DUBLURRING_OptimalValley} displays how $\RRE(\alpha,k)$ varies with $\alpha$ for a fixed $k$. 
\begin{figure}[htbp]
\centering
\begin{tabular}{cc}
\vspace{0.2cm}
{\small {(a)}} & {\small {(b)}}\\ 
\includegraphics[width=5.5cm]{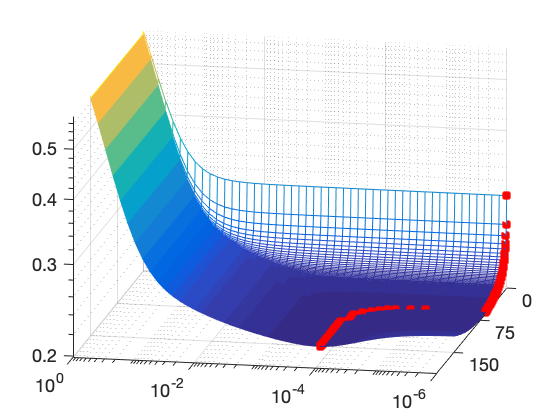} & 
\includegraphics[width=5.5cm]{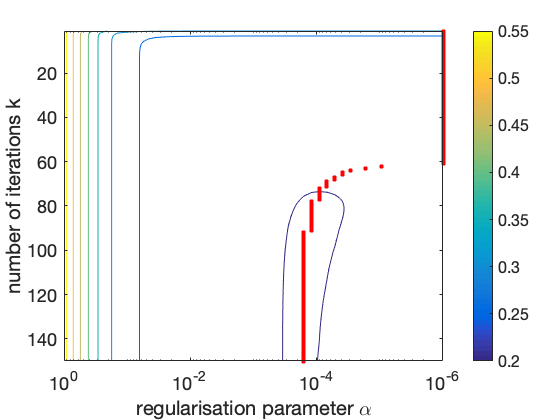}
\end{tabular}
\caption{\emph{Error surface} for Example 1. 50 logarithmically equispaced values of $\alpha$ between $10^{-6}$ and $1$ and $k=1,...,150$ are sampled. The markers highlight the combinations of the sampled values of $\alpha$ and $k$ leading the minimal relative reconstruction errors, for all $k=1,...,150$.}
\label{fig:DUBLURRING_OptimalValley}
\end{figure}

\paragraph{Example 2} 
A test problem modelling X-ray tomography involving the \emph{Shepp-Logan} phantom of size $256\times 256$ pixels, acquired through a parallel beam geometry consisting of 362 equidistant parallel beams rotated around 224 equidistant angles between $1^{\circ}$ and $180^{\circ}$ is considered. Gaussian white noise of level $\|e\|/\|b\|=10^{-2}$ is added to the measurements. This test problem is generated using the following instructions within \emph{IR Tools}:
\[
\text{\texttt{optn.angles = 1:0.8:180; [A,b,x] = PRtomo(256, optn); bn = PRnoise(b);}}
\]
The overdetermined coefficient matrix $A$ so computed has size $81088\times 65536$. Similarly to Example 1, the performance of Algorithm \ref{alg:new} is evaluated by comparing the relative restoration error $\RRE(\alpha_k,k)$ and the value of the regularization parameter $\alpha_k$ 
(computed at each iteration $k$) to the optimal ones and to the ones obtained running traditional hybrid methods (Algorithm \ref{alg:A2}); all the parameter choice rules listed in Table \ref{tab:rules} are tested. Graphs showing these comparisons are displayed in Figure \ref{fig:TOMO_graphs}. 
\begin{figure}[ht!]
\centering
\begin{tabular}{cccc}
\vspace{0.2cm}
{\small {(1a)}} & {\small {(1b)}} & {\small {(1c)}} & {\small {(1d)}} \\ 
\includegraphics[width=3cm]{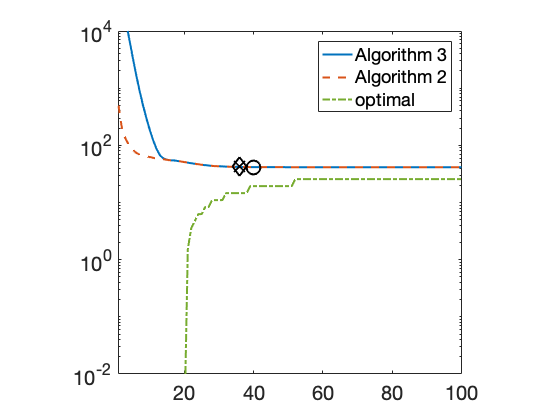} & 
\includegraphics[width=3cm]{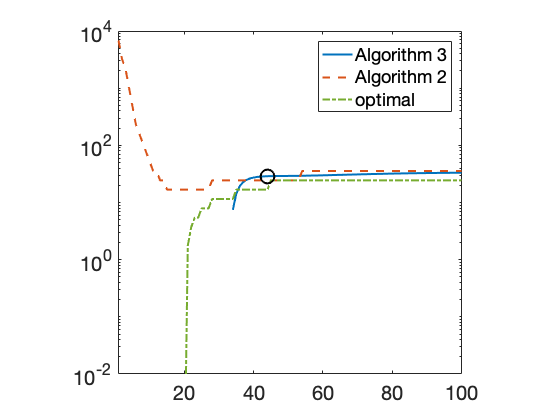} & 
\includegraphics[width=3cm]{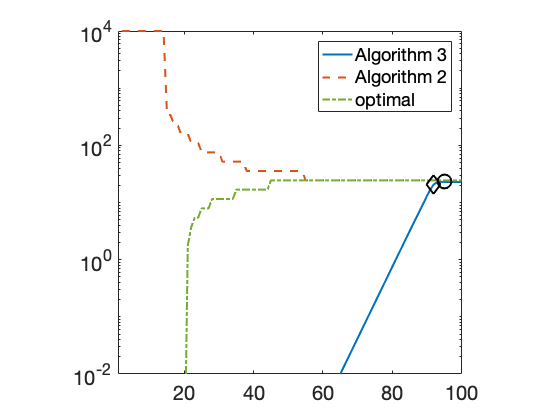}
&
\includegraphics[width=3cm]{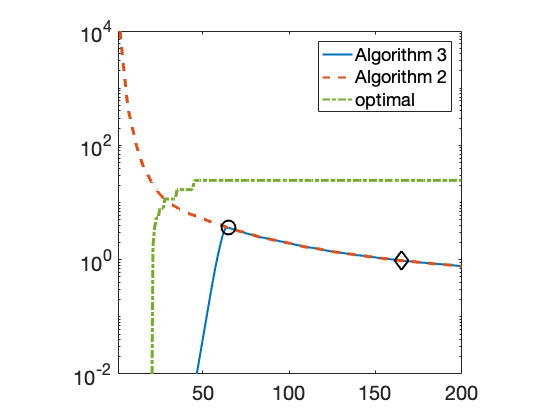} \\
{\small {(2a)}} & {\small {(2b)}} & {\small {(2c)}} & {\small {(2d)}} \\ 
\includegraphics[width=3cm]{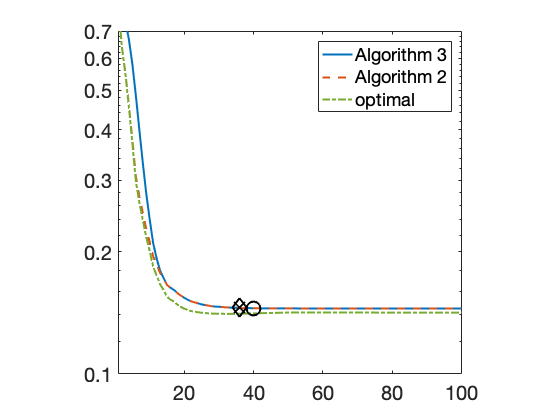} & 
\includegraphics[width=3cm]{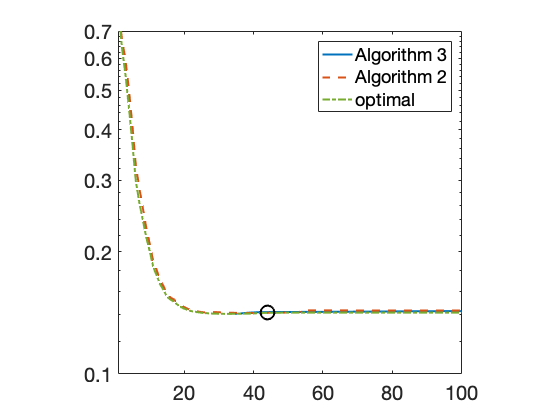} & 
\includegraphics[width=3cm]{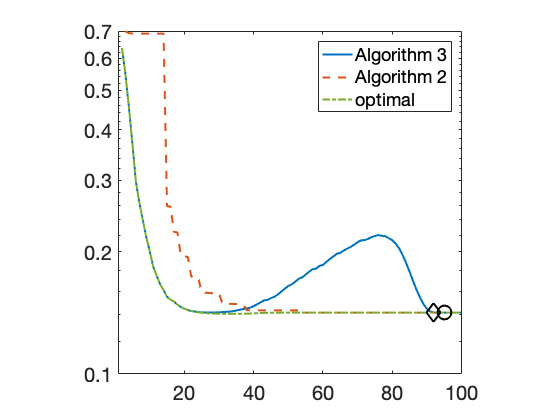}
&
\includegraphics[width=3cm]{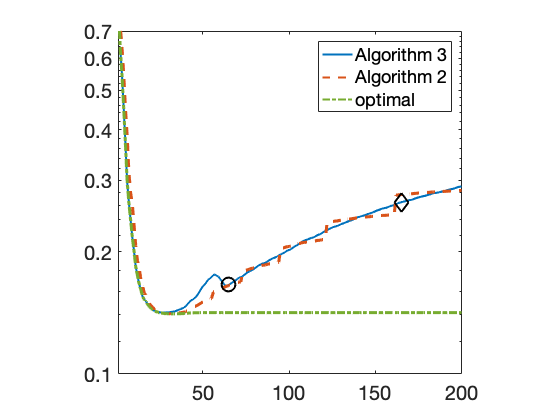}
\end{tabular}
\caption{Example 2. First row: values of the regularization parameter $\alpha_k$ versus number of iterations $k$. Second row: values of the relative restoration errors $\RRE(\alpha_k,k)$ versus number of iterations $k$. The following parameter choice rules are considered: (a) discrepancy principle; (b) GCV; (c) quasi-optimality criterion; (d) Regi\'{n}ska criterion. Special markers highlight the stopping iterations (with the conventions explained in Table \ref{tab:markers}).}
\label{fig:TOMO_graphs}
\end{figure}
%
Looking at Figure \ref{fig:TOMO_graphs} it is evident that 
the values of the parameter $\alpha_k$ computed by Algorithm \ref{alg:new} converge to the values obtained by applying a hybrid method. 
This can also be observed in Figure \ref{fig:TOMO_Valleys}, which displays the \emph{higher-level surface} for this test problem and for the parameter choice rules listed in Table \ref{tab:rules}. 
Special markers are used to highlight the values corresponding to the couples $(\alpha_k,k)$ computed by Algorithms \ref{alg:A2} and \ref{alg:new}, and the optimal values. Looking at Figure \ref{fig:TOMO_Valleys} it is evident that the shape of these surfaces stabilizes already for values of $k$ that are typically very small with respect to the dimension of the full-size problem. Note also that the GCV curves are very flat for combinations of small values of $k$ and small values of $\alpha$, while they quickly get steeper when $k$ is increased (see Figure \ref{fig:TOMO_Valleys} (b)): because of this, $k^\ast$ GKB iterations should be performed before starting the modified Newton method at line 4 of Algorithm \ref{alg:new} (see Section \ref{sec:GCV} for additional details).
\begin{figure}[ht]
\centering
\begin{tabular}{cc}
\vspace{0.2cm}
{\small {(a)}} & {\small {(b)}}\\ 
\includegraphics[width=5.5cm]{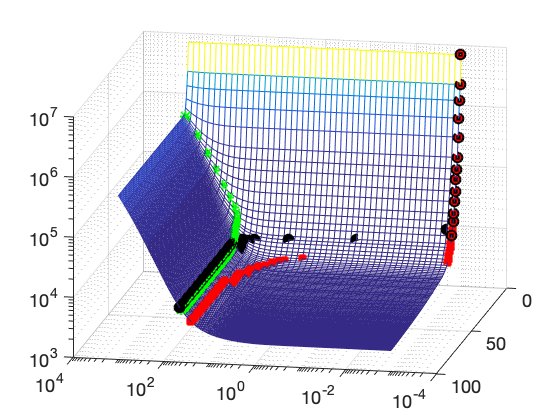} & 
\includegraphics[width=5.5cm]{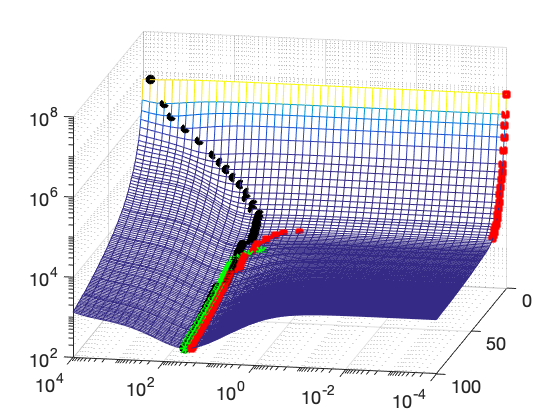} \\
{\small {(c)}} & {\small {(d)}}\\ 
\includegraphics[width=5.5cm]{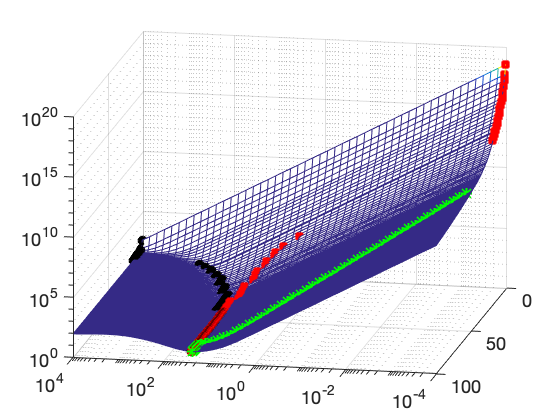} & 
\includegraphics[width=5.5cm]{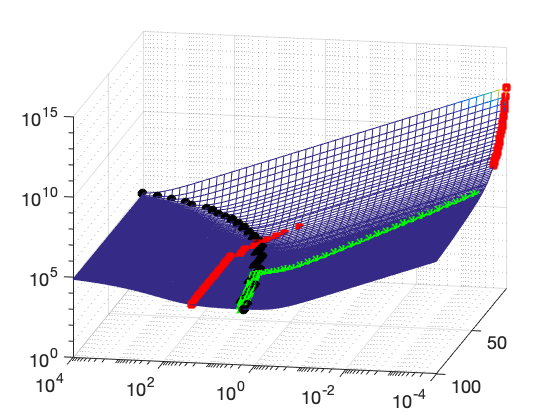}
\end{tabular}
\caption{\emph{Higher-level surface} for Example 2. 50 logarithmically equispaced values of $\alpha$ between $10^{-4}$ and $10^4$ and $k=1,...,100$ are considered. (a) Discrepancy principle (i.e., $\|b-Ax_k(\alpha)\|^2$); (b) GCV (i.e., upper bounds $P_k(\alpha)$ in (\ref{eqn:GCVtrace})); (c) quasi-optimality criterion (i.e., upper bounds $P_k(\alpha)$ in (\ref{qoup})); (d) Regi\'{n}ska criterion (i.e., upper bounds $P_k(\alpha)$ in (\ref{rup})). The green stars and black circles highlight the quantities computed by Algorithm \ref{alg:new} and Algorithm \ref{alg:A2}, respectively; the red squares denote the optimal quantities.} \label{fig:TOMO_Valleys}
\end{figure}
In most cases, the stopping criteria proposed in Section \ref{sec:new} succeed in stopping the iterations of Algorithm \ref{alg:new} when a good regularization parameter is computed (i.e., at a point that is quite close to the optimal one); 
when using Regi\'{n}ska criterion, stopping rule (48) fails (i.e., the method does not stop within the maximum number of performed iterations; see frame (d) of Figure \ref{fig:TOMO_Valleys}). 
\begin{figure}
\centering
\begin{tabular}{cc}
\vspace{0.2cm}
{\small {(a)}} & {\small {(b)}}\\ 
\includegraphics[width=5.5cm]{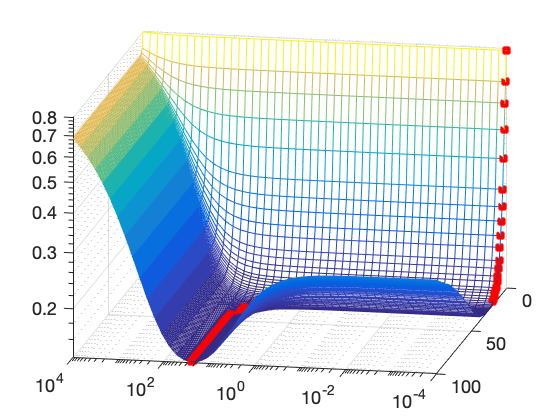} & 
\includegraphics[width=5.5cm]{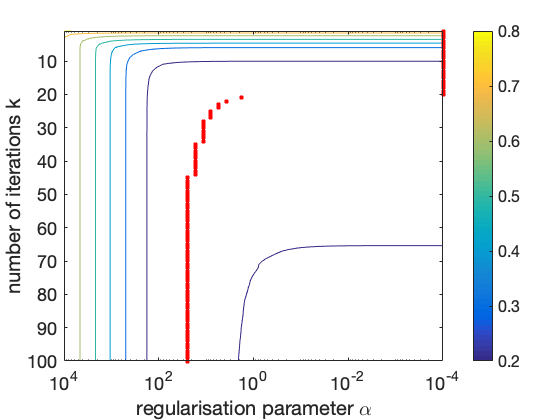}
\end{tabular}
\caption{\emph{Error surface} for Example 2. 50 logarithmically equispaced values of $\alpha$ between $10^{-4}$ and $10^4$ and $k=1,...,100$ are sampled. The markers highlight the combinations of the sampled values of $\alpha$ and $k$ leading the minimal relative reconstruction errors, for all $k=1,...,100$.}
\label{fig:TOMO_OptimalValley}
\end{figure}
Finally, Figure \ref{fig:TOMO_OptimalValley} displays the \emph{error surface} for this test problem.
The behavior of the optimal value of the regularization parameters at each iteration (highlighted by a special marker) is analogous to the one displayed in Figure \ref{fig:DUBLURRING_OptimalValley} (and in particular, the the value of the optimal regularization parameter $\alpha$ at each iteration $k$ stabilizes after a small number of iterations).

\section{Conclusions and future work}\label{sec:end}
This paper described and analyzed a new class of algorithms (Algorithm \ref{alg:new}) for the solution of bilevel optimization problems (\ref{GenP}) arising when simultaneously computing a Tikhonov-regularized solution and a regularization parameter according to a given rule, in the framework of large-scale linear inverse problems. By a novel use of Krylov projection methods based on the GKB algorithm, its connections with Gaussian quadrature rules, and a modified Newton method, the proposed approach ``interlaces'' the iterations performed to apply a given (nonlinear) parameter choice rule and the iterations performed to iteratively solve the (linear) Tikhonov-regularized problem, giving rise to an efficient and principled strategy that delivers results comparable to the ones obtained with well-established solvers (e.g., traditional hybrid methods).

Future work includes the natural extension of the new class of algorithms to work with Krylov projection methods that are based on algorithms other than GKB (e.g., the Arnoldi algorithm or flexible Krylov methods; see \cite{flexi}): while the computations involved in Algorithm \ref{alg:new} can be adapted to these situations, the theoretical analysis of the resulting strategies needs to be carefully rethought. Moreover, the new class of algorithms can be extended to handle Tikhonov-TSVD regularization, i.e., regularization methods that apply Tikhonov method to a TSVD-projected linear system (see, e.g., \cite{Renaut2017}): in these cases, one should replace the Krylov space $\Kry_k$ appearing in Algorithm \ref{alg:new} by the space spanned by the first $k$ right singular vectors of $A$; a careful theoretical analysis would be needed to prove convergence results. Also, other parameter choice strategies that can be expressed in the framework of bilevel optimization problems (e.g., the UPRE criterion, see \cite{Renaut2017} and the references therein) can be considered. To conclude, the framework of Algorithm \ref{alg:new} is very general, and can be potentially extended to a variety of bi-level optimization methods that involve the solution of a nonlinear higher-level problem and a linear lower-level problem. 

\bibliographystyle{plain}      
\bibliography{biblio}  
\end{document}